\documentclass[a4paper,11pt]{amsart}

\usepackage{amsfonts,latexsym,rawfonts,amsmath,amssymb,amsthm, mathrsfs, lscape}
\usepackage[all]{xy}
\usepackage[english]{babel}
\usepackage[utf8]{inputenc}
\usepackage[T1]{fontenc}
\usepackage{graphicx}
\usepackage{booktabs}
\usepackage{array, tabularx}
\usepackage{setspace}
\usepackage{textcomp}
\usepackage{tikz}
\usepackage{multirow}
\usepackage{paralist}
\usepackage{lscape}
\usepackage{placeins}
\usepackage{soul}

\usepackage[hypertexnames=false,
backref=page,
    pdftex,
    pdfpagemode=UseNone,
    breaklinks=true,
    extension=pdf,
    colorlinks=true,
    linkcolor=blue,
    citecolor=blue,
    urlcolor=blue,
]{hyperref}

\renewcommand{\Re}{\mathsf{Re}\,}
\renewcommand{\Im}{\mathsf{Im}\,}

\newcommand{\referenza}{}

\newtheorem{thm}{Theorem}[section]
\newtheorem*{thm*}{Theorem \referenza}
\newtheorem{cor}[thm]{Corollary}
\newtheorem*{cor*}{Corollary \referenza}
\newtheorem{lem}[thm]{Lemma}
\newtheorem*{lem*}{Lemma \referenza}

\newtheorem{prop}[thm]{Proposition}
\newtheorem*{prop*}{Proposition \referenza}

\newtheorem*{conj*}{Conjecture \referenza}
\newtheorem{rmk}[thm]{Remark}

\newtheorem{exa}[thm]{Example}

\numberwithin{equation}{section}

\def \N {\mathbb N}

\def \Q {\mathbb Q}
\def \R {\mathbb R}
\def \C {\mathbb C}
\def \Z {\mathbb Z}

\newcommand{\sspace}{\text{\--}}
\newcommand{\ssspace}{\text{\textdblhyphen}}

\DeclareMathOperator{\imm}{im}

\allowdisplaybreaks[1]

\title[lcs cohomologies]{Cohomologies of locally conformally symplectic manifolds and solvmanifolds}

\author{Daniele Angella}
\address[Daniele Angella]{
Dipartimento di Matematica e Informatica ``Ulisse Dini''\\
Universit\`a degli Studi di Firenze\\
viale Morgagni 67/a\\
50134 Firenze, Italy
}
\email{daniele.angella@gmail.com}
\email{daniele.angella@unifi.it}

\author{Alexandra Otiman}
\address[Alexandra Otiman]{
Institute of Mathematics "Simion Stoilow" of the Romanian Academy, 21,
Calea Grivitei Street, 010702, Bucharest, Romania
}
\address[Alexandra Otiman]{
University of Bucharest, Faculty of Mathematics and Computer Science, 14
Academiei Str., Bucharest, Romania.
}
\email{alexandra\_otiman@yahoo.com}

\author{Nicoletta Tardini}
\address[Nicoletta Tardini]{Dipartimento di Matematica\\
Universit\`a di Pisa\\
largo Bruno Pontecorvo 5\\
56127 Pisa, Italy
}
\email{nicoletta.tardini@gmail.com}
\email{tardini@mail.dm.unipi.it}

\thanks{The first author is supported by the SIR2014 project RBSI14DYEB ``Analytic aspects in complex and hypercomplex geometry'', by ICUB Fellowship for Visiting Professor, and by GNSAGA of INdAM.
The third author is supported by Project PRIN ``Varietà reali e complesse: geometria, topologia e analisi armonica'' and by GNSAGA of INdAM}

\date{\today}

\begin{document}

\dedicatory{
Dedicated to Professor Paolo Piccinni on the occasion of his 65th birthday.\\
Buon compleanno!
}

\begin{abstract}
We study the Morse-Novikov cohomology and its almost-symplectic counterpart on manifolds admitting locally conformally symplectic structures.
More precisely, we introduce lcs cohomologies and we study elliptic Hodge theory, dualities, Hard Lefschetz Condition.
We consider solvmanifolds and Oeljeklaus-Toma manifolds. In particular, we prove that Oeljeklaus-Toma manifolds with precisely one complex place, and under an additional arithmetic condition, satisfy the Mostow property. This holds in particular for the Inoue surface of type $S^0$.

\end{abstract}

\maketitle

\section*{Introduction}
On a compact differentiable manifold $X$, {\em flat line bundles} (namely, local systems of $1$-dimensional $\C$-vector spaces,) are determined by the associated monodromy homomorphism $\pi_1(X,x)\to \C^\times$, which can be viewed as a cohomology class $[\vartheta] \in H^1(X;\C)$. Consider the twisted differential $d_\vartheta:=d-\vartheta\wedge\sspace$, that is the exterior derivative perturbed by a closed $1$-form $\vartheta$. The cohomology of the perturbed complex $(\wedge^\bullet X, d_\vartheta)$ is called {\em Morse-Novikov cohomology} $H^\bullet_{\vartheta}(X)$ \cite{novikov-1, novikov-2, guedira-lichnerowicz} of $X$ with respect to $\vartheta$, and it depends just on $[\vartheta]\in H^1(X;\R)$ up to gauge equivalence. It may provide informations on the manifold itself. See {\itshape e.g.} the explicit computations on Inoue surfaces in \cite{otiman-2}, where the Morse-Novikov cohomology allows to distinguish between Inoue surfaces of type $S^+$ and $S^-$, even if they have the same Betti numbers. So, it may be useful to understand the cohomology $H^\bullet_{\vartheta}(X)$ varying $[\vartheta]\in H^1(X;\R)$; in particular one can study, for example, $H^\bullet_{k\cdot \vartheta}(X)$ varying $k\in\R$ for a fixed $[\vartheta]\in H^1(X;\R)$.

In the holomorphic category, twisted differentials have been studied in \cite{kasuya-imrn}, see also \cite{angella-kasuya-4}. In particular, H. Kasuya gives in \cite[Theorem 1.7]{kasuya-imrn} a structure theorem for K\"ahler solvmanifolds in terms of strong-Hodge-decomposition with respect to {\em any} perturbation of the differentials, which he calls {\em hyper-strong-Hodge-decomposition}. This result yields a Hodge-theoretical proof of the Arapura theorem characterizing solvmanifolds in class $\mathcal{C}$ of Fujiki, see \cite[Theorem 3.3]{angella-kasuya-4}.

The twisted differential $d_\vartheta$ has also a geometric description. In fact, by the Poincaré Lemma, closed $1$-forms correspond to local conformal changes. So, for example, for an almost-symplectic form $\Omega$, (that is, a non-degenerate $2$-form,) the {\em locally conformal symplectic} condition corresponds to $d_\vartheta\Omega=0$ for some closed {\em Lee form} $\vartheta$, while the symplectic condition corresponds to $d\Omega=0$, that is the case $\vartheta=0$.

In this note, we consider locally conformal symplectic (say, lcs) structures. We take their associated closed Lee forms as natural twists for the differential, --- in the spirit of the equivariant point of view introduced in \cite{gini-ornea-parton-piccinni}.
We introduce and study cohomologies in the lcs setting as analogues of the Tseng and Yau symplectic cohomologies \cite{tseng-yau-1, tseng-yau-2}. We develop here the algebraic aspects arising from a structure of bi-differential vector space, while H. V. Le and J. Van\u{z}ura study primitive cohomology groups in \cite{le-vanzura}.
(See also \cite{angella-kasuya-3}, where symplectic cohomologies and symplectic cohomologies with values in a local system are studied, with focus on solvmanifolds.)

More precisely, under the inspiration of \cite{brylinski, yan}, we start by looking at the commutation between the twisted differential $d_\vartheta$ by the Lee form and the $\mathfrak{sl}(2;\R)$-representation operators associated to the lcs (almost-symplectic is enough) form $\Omega$, namely, $L:=\Omega\wedge \sspace$ and $\Lambda:=-\iota_{\Omega^{-1}}$ and $H=[L,\Lambda]$. It is clear that $d_\vartheta L=Ld+d_\vartheta\Omega$; so, the lcs condition $d_\vartheta\Omega=0$ assures that $d_\vartheta L=Ld$. Moreover, the commutation between $d_\vartheta$ and $\Lambda$ was computed in \cite[Proposition 2.8]{angella-ugarte-1}, and once again it gives a change of the twist but still in the same line; see also \cite[Section 2]{le-vanzura}. Both these results suggest to look not only at the twist $[\vartheta]$, but also at $k\cdot[\vartheta]$ varying $k\in\R$.
Moreover, in the spirit of the Novikov inequalities, which link the number of zeroes of a closed $1$-forms of Morse-type with the dimension of the Morse-Novikov cohomology, note that $\vartheta$ and $k\cdot\vartheta$ have the same zeroes when $k\in\R\setminus\{0\}$. For large $k$, interesting phenomena occour: {\itshape e.g.} if $\vartheta$ is not exact, then $k\cdot \vartheta$ is the Lee forms of a lcs structure \cite{eliashverg-murphy}; if $\vartheta$ is nowhere vanishing, then the Morse-Novikov cohomology with respect to $k\cdot \vartheta$ vanishes \cite{pajitnov}.

This is our motivation to define a bi-differential graded vector space associated to $(k+\Z)\cdot[\vartheta]$, see Lemma \ref{lemma:double-cplx}.
Once we have this bi-differential vector space structure, we investigate its associated cohomologies: other than the Morse-Novikov cohomology and its lcs-dual, we have {\em lcs-Bott-Chern and Aeppli} cohomologies. Following the same pattern as \cite{brylinski, mathieu, yan, merkulov, guillemin, cavalcanti, tseng-yau-1}, we study elliptic-Hodge-theory, and we get some results concerning Poincaré dualities, see Proposition \ref{prop:poincare} and Theorem \ref{thm:poincare-ABC}, and Hard Lefschetz Condition, see Theorem \ref{thm:lcs-HLC} and Theorem \ref{thm:hlc-lcslemma}. Finally, we study some explicit examples, on nilmanifolds (Kodaira-Thurston surface \cite{kodaira-surfaces-1, thurston}) and {\em solvmanifolds} (Inoue surfaces of type $S^+$ \cite{inoue}, for which see also \cite{otiman-2}, and Oeljeklaus-Toma manifolds \cite{oeljeklaus-toma}).

For compact quotients of connected simply-connected completely solvable Lie groups, the Hattori theorem \cite[Corollary 4.2]{hattori} allows to reduce the computation of the Morse-Novikov cohomology at the linear level of the Lie algebra, and the same holds for lcs cohomologies, see Lemma \ref{lem:inv-lcs-cohom}.
In general, for a solvmanifold $\Gamma\backslash G$ which is not completely solvable, there is no reason of having $H^{*}(\mathfrak{g}) \simeq H^{*}(\Gamma\backslash G)$. One situation when this happens is when the solvmanifold satisfies the {\em Mostow condition} \cite{mostow}. We prove this condition suffices also for the lcs cohomologies with respect to an invariant closed one-form, see Proposition \ref{prop:mostow}. The case of Inoue surfaces is interesting because two subclasses, $S^{\pm}$, are completely-solvable, falling thus under the scope of the Hattori theorem; however this is not the case of the subclass $S^0$. In \cite{otiman-2}, the computations of the cohomology are done without using the structure of solvmanifold, but instead with a "twisted" version of the Mayer-Vietoris sequence. We prove here that Inoue surfaces of type $S^0$ and, more in general, certain Oeljeklaus-Toma manifolds with precisely one complex place satisfy the Mostow condition, see Proposition \ref{prop:mostow-inoue} and Theorem \ref{thm:mostow-ot} respectively. More precisely, here we have to assume an arithmetic condition on the associated number field, namely, that there is no totally real intermediate extension. This holds for example for the Inoue surface of type $S^0$, that is, in the case $(s,t)=(1,1)$, see also Proposition \ref{prop:mostow-inoue}. As we show in Proposition \ref{prop:example-mostow}, for any $s$ there exists an Oeljeklaus-Toma manifold of type $(s,1)$ satisfying such a property.

\bigskip

\noindent{\sl Acknowledgments.}
The authors would like to thank Giovanni Bazzoni, Liviu Ornea, Luis Ugarte, Victor Vuletescu,  for interesting discussions.
The first-named and the third-named authors would like to thank also Adriano Tomassini for his constant support and encouragement and for useful discussions. The second-named author is also grateful for constructive discussions to Andrei Sipo\c s and Miron Stanciu and would like to thank Liviu Ornea for his constant guidance. Part of this work has been done during the stay of the first-named author at Universitatea din Bucure\c{s}ti with the support of an ICUB Fellowship: he would like to thank Liviu Ornea and Victor Vuletescu for the invitation, and the whole Department for the warm hospitality.

\section{Bi-differential graded vector space for lcs structures}

Let $X$ be a compact differentiable manifold endowed with a {\em locally conformal symplectic form} $\Omega$ with {\em Lee form} $\vartheta$, namely: $\Omega$ is an almost-symplectic form ({\itshape i.e.} a non-degenerate $2$-form) such that
$$ d\Omega - \vartheta\wedge\Omega \;=\; 0 \qquad \text{ with } \qquad d\vartheta \;=\; 0 \;. $$

We set
$$L \;:=\;\Omega\wedge\sspace \qquad \text{ and } \qquad \Lambda\;:=\;-\iota_{\Omega^{-1}} \;,$$
where $\iota$ denotes the contraction.
Read $\Lambda=-L^\star=-\star^{-1}L\star$, up to a sign, as the symplectic adjoint of $L$, namely, the dual of $L$ with respect to the $L^2$-pairing induced by the almost-symplectic form $\Omega$.
Recall that, $L$ and $\Lambda$ together with
$$H\;:=\;[L,\Lambda] \;,$$
yield an $\mathfrak{sl}(2;\R)$-representation on $\wedge^\bullet X$, see \cite[Corollary 1.6]{yan}, see also \cite[Corollary 2.4]{le-vanzura} quoting \cite[Section 1]{lycagin}.

For $k\in\R$, we consider the following operators, compare \cite[Section 2]{le-vanzura}:
\begin{eqnarray*}
d_k &:=& d_{k\vartheta} \;:=\; d-(k\vartheta)\wedge\sspace \colon \wedge^\bullet X \to \wedge^{\bullet+1}X \;, \\[5pt]
\delta_k &:=& d_{k-1}\Lambda - \Lambda d_{k} \colon \wedge^{\bullet}X\to\wedge^{\bullet-1}X \;.
\end{eqnarray*}
By a straightforward computation, the Leibniz rule for $d_k$ reads as:
$$ d_k(\alpha\wedge\beta) = d_{k-h} \alpha \wedge\beta+(-1)^{\deg\alpha}\alpha\wedge d_h\beta , $$
for $h\in\R$, see \cite[Lemma 2.1]{le-vanzura}.
We also notice that, if we change $\vartheta$ by $\vartheta+df$, then the lcs structure $\Omega$ with Lee form $\vartheta$ yields the lcs structure $\exp(f)\Omega$ with respect to the Lee form $\vartheta+df$, and the above operators change as follows:
\begin{eqnarray}
d_{k(\vartheta+df)} &=& \exp(kf)\, d_{k\vartheta}\, \exp(-kf) , \label{eq:gauge-d} \\[5pt]
\delta_{k(\vartheta+df)} &=& \exp((k-1)f)\, \delta_{k\vartheta}\, \exp(-kf) . \label{eq:gauge-delta}
\end{eqnarray}

\begin{rmk}
Note that, in \cite{le-vanzura}, the sign of $\vartheta$ is chosen opposite: $d^{\text{LV}}_k:=d+k\vartheta\wedge\sspace$. Therefore we have $d_{k}^{\text{LV}}=d_{-k}$. Their second operator is $\delta^{\text{LV}}_k\lfloor_{\wedge^h X}:=(-1)^h\star d_{n+k-h}^{\text{LV}}\star$, \cite[Equation (2.11)]{le-vanzura}, that is, $\delta^{\text{LV}}_k\lfloor_{\wedge^hX}=\delta_{2n+k-2h}$, as follows by the formulas \eqref{eq:dk-star} and \eqref{eq:delta-dual-d} below.
Moreover, as for $\Lambda$, the notation in our note differs from \cite{le-vanzura} up to a sign.
\end{rmk}

In order to give a different interpretation of $\delta_k$,
we need some preliminaries.
Recall that, once fixed any almost-complex structure $J$ on $X$, one defines $d_k^c:=J^{-1}d_kJ$.
Denoting with $*$ the Hodge-$*$-operator associated to
a fixed $J$-Hermitian metric $g$ on $X$,
the formula for the adjoint of $d_k$, respectively $d^c_k$, with respect to the $L^2$-pairing induced by $g$ is $d_k^*=-*d_{-k}*$, respectively $(d^c_k)^*=-*d^c_{-k} *$.
Moreover, we can also consider the $L^2$-pairing induced by the almost-symplectic structure $\Omega$, whence the symplectic Hodge-$\star$-operator in \cite[Section 2]{brylinski}. The analogue formulas for the adjoint in the symplectic context are $d_k^\star\lfloor_{\wedge^hX}=(-1)^h\star d_{-k}\star$, and $(d^c_k)^\star\lfloor_{\wedge^hX}=(-1)^h\star d^c_{-k} \star$. (Recall that $*^2\lfloor_{\wedge^h X}=(-1)^h\cdot\mathrm{id}$ and $\star^2=\mathrm{id}$.)
Finally, recall that: if $J$ is an almost-complex structure compatible with the almost-symplectic form $\Omega$, once set $g:=\Omega(\sspace,J\ssspace)$ the corresponding $J$-Hermitian metric, (that is, $(g,J,\Omega)$ is an almost-Hermitian structure,) then we have the relation $\star=J*$ \cite[Corollary 2.4.3]{brylinski}. Therefore, we get
\begin{equation}\label{eq:dk-star}
d_k^\star\lfloor_{\wedge^hX}\;=\;(-1)^h\star d_{-k}\star\;=\;-*d_{-k}^c*\;=\;(d_k^c)^* \;.
\end{equation}
We have the following.

\begin{lem}[{\cite[Proposition 2.8]{angella-ugarte-1}}]
 Let $X$ be a compact differentiable manifold of dimension $2n$, endowed with a locally conformal symplectic form $\Omega$ with Lee form $\vartheta$.
 Consider an almost-complex structure compatible with $\Omega$, and the associated Hermitian metric.
 Then
\begin{equation}\label{eq:delta-dual-d}
\left. \delta_k \right.\lfloor_{\wedge^hX} \;=\; d_{-(n+k-h)}^{\star} \;=\; (d^c_{-(n+k-h)})^* \;.
\end{equation}
\end{lem}

We have
\begin{lem}\label{lemma:double-cplx}
 Let $X$ be a compact differentiable manifold of dimension $2n$, and let $\vartheta$ be a $d$-closed $1$-form. Assume that there is a locally conformal symplectic form $\Omega$ with Lee form $\vartheta$.
 Then, for any fixed $k\in\R$, the diagram
\begin{equation}\label{eq:double-cplx} 
\xymatrix{
         & \vdots   & \vdots        & \vdots        &        \\
  \cdots \ar[r] & \wedge^{h-2}X \ar[r]^{d_{k-1}} \ar[u] & \wedge^{h-1}X \ar[r]^{d_{k-1}} \ar[u]& \wedge^{h  }X \ar[u] \ar[r] & \cdots   \\
  \cdots \ar[r] & \wedge^{h-1}X \ar[r]^{d_k} \ar[u]^{\delta_k} & \wedge^{h}  X \ar[r]^{d_k} \ar[u]^{\delta_k} & \wedge^{h+1}X \ar[r] \ar[u]^{\delta_{k}} & \cdots \\
  \cdots \ar[r] & \wedge^hX \ar[r]^{d_{k+1}} \ar[u]^{\delta_{k+1}}   & \wedge^{h+1}X \ar[u]^{\delta_{k+1}} \ar[r]^{d_{k+1}} & \wedge^{h+2}X \ar[u]^{\delta_{k+1}} \ar[r] & \cdots  \\
         & \vdots\ar[u]    & \vdots        \ar[u] & \vdots \ar[u]    &        \\
 }
\end{equation}
represents a $\Z$-graded bi-differential vector space.
\end{lem}

\begin{proof}
We have to prove that:
$$ (d_k)^2 \;=\; 0 \;, \qquad \delta_k \delta_{k+1} \;=\; 0 \;, \qquad d_{k-1}\delta_{k}+\delta_{k}d_k \;=\; 0 \;. $$
\begin{itemize}
 \item More in general, by straightforward computations, we notice that
 $$ d_k d_\ell \;=\; (\ell-k)\vartheta\wedge \sspace \;. $$
 \item Let $J$ be an almost-complex structure compatible with the almost-symplectic structure $\Omega$, and let $g$ be the associated $J$-Hermitian metric. We compute:
 \begin{eqnarray*}
  \lefteqn{ \left. (\delta_k \delta_{k+1}) \right\lfloor_{\wedge^h X} } \\[5pt]
  &=& (d^c_{-(n+k-(h-1))})^*(d^c_{-(n+(k+1)-h)})^* \\[5pt]
  &=& * J^{-1} d_{n+k-h+1} J * * J^{-1} d_{n+k-h+1} J * \\[5pt]
  &=& (-1)^{h+1} * J^{-1} d_{n+k-h+1} d_{n+k-h+1} J * \;=\; 0 \;.
 \end{eqnarray*}
 The third equality follows from the fact that $*^2\lfloor_{\wedge^hX}=(-1)^h$;
 the last one follows by the previous point of the proof.
 \item We compute:
 \begin{eqnarray*}
  \lefteqn{ d_{k-1}\delta_k+\delta_kd_k } \\[5pt]
  &=& d_{k-1}d_{k-1}\Lambda-d_{k-1}\Lambda d_k+d_{k-1}\Lambda d_k-\Lambda d_k d_k
  \;=\; 0 \;.
 \end{eqnarray*}
\end{itemize}
This completes the proof.
\end{proof}

\section{Cohomologies for lcs structures}

Let $X$ be a compact differentiable manifold, and let $\vartheta$ be a $d$-closed $1$-form. Assume that there exits a locally conformal symplectic form $\Omega$ on $X$ with Lee form $\vartheta$, namely, $\Omega$ is a non-degenerate $2$-form such that $d_{\vartheta}\Omega=0$. Fix $k\in\R$. Once given the bi-differential $\Z$-graded vector space in the Lemma \ref{lemma:double-cplx}, we can define the following cohomologies:
\begin{eqnarray*}
 H^{\bullet}_{d_{k}}(X) \;:=\; \frac{\ker d_k}{\imm d_k} \;,
 & \qquad &
 H^{\bullet}_{\delta_k}(X) \;:=\; \frac{\ker \delta_k}{\imm \delta_{k+1}} \;,
 \\[5pt]
 H^{\bullet}_{d_k+\delta_{k}}(X) \;:=\; \frac{\ker d_k \cap \ker \delta_k}{\imm \delta_{k+1}d_{k+1}} \;,
 & \qquad &
 H^{\bullet}_{\delta_kd_k}(X) \;:=\; \frac{\ker \delta_kd_k}{\imm d_k + \imm \delta_{k+1}} \;.
\end{eqnarray*}
We call $H^{\bullet}_{d_k+\delta_{k}}(X)$ the \emph{lcs-Bott-Chern cohomology} of weight $k$ of $X$, and $H^{\bullet}_{\delta_kd_k}(X)$ the \emph{lcs-Aeppli cohomology} of weight $k$ of $X$.
Note that, thanks to \eqref{eq:gauge-d} and \eqref{eq:gauge-delta}, the above cohomologies depend just on $[\vartheta]\in H^1_{dR}(X;\R)$, up to gauge equivalence.

The identity induces natural maps of $\Z$-graded vector spaces:
\begin{equation}\label{eq:coom}
\xymatrix{
 & H^{\bullet}_{d_k+\delta_k}(X) \ar[dl] \ar[dr] \ar[dd] & \\
 H^{\bullet}_{d_{k}}(X) \ar[rd] & & H^{\bullet}_{\delta_{k}}(X) \ar[ld] \\
 & H^{\bullet}_{\delta_kd_{k}}(X) &
}
\end{equation}

By definition, we say that $X$ {\em satisfies the $\delta_kd_k$-Lemma} if the natural map $H^\bullet_{d_k+\delta_k}(X) \to H^\bullet_{\delta_k d_k}(X)$ induced by the identity is injective.
We say that $X$ {\em satisfies the lcs-Lemma} if it satisfies the $\delta_kd_k$-Lemma for any $k\in\R$. In this case, all the above maps are isomorphisms, see \cite[Lemma 5.15]{deligne-griffiths-morgan-sullivan}, adapted in \cite[Lemma 1.4]{angella-tomassini-JNC} to the $\Z$-graded case.

\medskip

\begin{rmk}[Comparison with Tseng and Yau's symplectic cohomologies]
 In the case $\vartheta=0$, the lcs form $\Omega$ with Lee form $\vartheta$ is in fact symplectic.
 In \cite{tseng-yau-1, tseng-yau-2}, Tseng and Yau introduce and study the Bott-Chern and the Aeppli cohomologies for symplectic manifolds, defined as
 $$ H^{\bullet}_{d+d^\Lambda}(X) \;:=\; \frac{\ker d \cap \ker d^\Lambda}{\imm d d^\Lambda} \qquad \text{ and } \qquad H^{\bullet}_{dd^\Lambda}(X) \;:=\; \frac{\ker dd^\Lambda}{\imm d + \imm d^\Lambda} \;, $$
 where $d^\Lambda := [d,\Lambda]$.
 In case $\vartheta=0$, notice that, for any $k\in\R$, one has $d_k=d$ and $\delta_k=d^\Lambda$, whence
 $$ H^{\bullet}_{d_k+\delta_k}(X) \;=\; H^\bullet_{d+d^\Lambda}(X) \;, \qquad
 H^{\bullet}_{\delta_kd_k}(X) \;=\; H^\bullet_{dd^\Lambda}(X) \;. $$
This means that the lcs-cohomologies defined above coincide with the ones
 defined by Tseng and Yau in the symplectic case.
In particular, $X$ satisfies the $\delta_kd_k$-Lemma for some $k$ if and only if it satisfies the lcs-Lemma if and only if the symplectic structure satisfies the Hard Lefschetz Condition, see \cite[Proposition 3.13]{tseng-yau-1} and the references therein.
\end{rmk}

\subsection{Elliptic Hodge theory for lcs cohomologies}
As before, consider an almost-complex structure $J$ compatible with the almost-symplectic form $\Omega$, and let $g:=\Omega(\sspace,J\ssspace)$ be the corresponding $J$-Hermitian metric. Fix $k\in\R$.
We consider the adjoint operators
$$ d_k^* \;=\; - * d_{-k} * \;, \qquad \left. \delta_k^* \right\lfloor_{\wedge^h X} \;=\; d^c_{-(n+k-h)} \;, $$
of $d_k$, respectively $\delta_k$, with respect to the $L^2$-pairing induced by $g$.

We follow \cite{kodaira-spencer-3, schweitzer, tseng-yau-1}, and we define the following operators, see also \cite{guedira-lichnerowicz} for the Morse-Novikov cohomology:
\begin{eqnarray*}
 \Delta_{d_k} &:=& d_kd_k^*+d_k^*d_k \;, \\[5pt]
 \Delta_{\delta_k} &:=& \delta_k^*\delta_k+\delta_{k+1}\delta_{k+1}^* \;, \\[5pt]
 \Delta_{d_k+\delta_k} &:=& d_k^*d_k+\delta_k^*\delta_k+ (\delta_{k+1}d_{k+1})(\delta_{k+1}d_{k+1})^*+(\delta_{k}d_{k})^*(\delta_{k}d_{k}) \\[5pt]
 && +(d_k^*\delta_{k+1})(d_k^*\delta_{k+1})^*+(d_{k-1}^*\delta_k)^*(d_{k-1}^*\delta_k) \;, \\[5pt]
 \Delta_{\delta_kd_k} &:=& d_kd_k^*+\delta_{k+1}\delta_{k+1}^*+(\delta_kd_k)^*(\delta_kd_k)+(\delta_{k+1}d_{k+1})(\delta_{k+1}d_{k+1})^* \\[5pt]
 && +(\delta_{k+1}d_{k+1}^*)(\delta_{k+1}d_{k+1}^*)^*+(d_k\delta_{k}^*)(d_k\delta_{k}^*)^* \;.
\end{eqnarray*}

\begin{prop}\label{prop:hodge-isom}
 Let $X$ be a compact differentiable manifold of dimension $2n$, and let $\vartheta$ be a $d$-closed $1$-form. Assume that there is a locally conformal symplectic form $\Omega$ with Lee form $\vartheta$.
Fix an almost-complex structure $J$ compatible with $\Omega$, and let $g$ be the corresponding $J$-Hermitian metric.
Fix $k\in\R$. Then:
 \begin{enumerate}
  \item[(i)] the operators $\Delta_{d_k}$, $\Delta_{\delta_k}$, $\Delta_{d_k+\delta_k}$, $\Delta_{\delta_kd_k}$ are differential self-adjoint elliptic operators;
  \item[(ii)] the following Hodge decompositions hold:
  \begin{eqnarray*}
   \wedge^\bullet X &=& \ker \Delta_{d_k} \oplus \imm \Delta_{d_k} \;, \\[5pt]
   \wedge^\bullet X &=& \ker \Delta_{\delta_k} \oplus \imm \Delta_{\delta_k} \;, \\[5pt]
   \wedge^\bullet X &=& \ker \Delta_{d_k+\delta_k} \oplus \imm \Delta_{d_k+\delta_k} \;, \\[5pt]
   \wedge^\bullet X &=& \ker \Delta_{\delta_kd_k} \oplus \imm \Delta_{\delta_kd_k} \;;
  \end{eqnarray*}
  \item[(iii)] the following isomorphisms hold:
  \begin{eqnarray*}
   \ker\Delta_{d_k} & \stackrel{\simeq}{\to} & H^\bullet_{d_k}(X) \;, \\[5pt]
   \ker\Delta_{\delta_k} & \stackrel{\simeq}{\to} & H^\bullet_{\delta_k}(X) \;, \\[5pt]
   \ker\Delta_{d_k+\delta_k} & \stackrel{\simeq}{\to} & H^\bullet_{d_k+\delta_k}(X) \;, \\[5pt]
   \ker\Delta_{\delta_kd_k} & \stackrel{\simeq}{\to} & H^\bullet_{\delta_kd_k}(X) \;;
  \end{eqnarray*}
  \item[(iv)] in particular, the lcs-cohomologies $H^{\bullet}_{d_k}(X)$, $H^{\bullet}_{\delta_k}(X)$, $H^{\bullet}_{d_k+\delta_k}(X)$, $H^{\bullet}_{\delta_kd_k}(X)$ have finite dimension.
 \end{enumerate}
\end{prop}

\begin{proof}
 Notice that the top order terms coincide with the terms corresponding to $k=0$. In particular, the operators are ellipic, see \cite[Proposition 3.3, Theorem 3.5, Theorem 3.16]{tseng-yau-1}. The statement follows from the general theory of differential self-adjoint elliptic operators.
\end{proof}

\subsection{Symmetries in lcs cohomologies}

The following two results resumes the dualities {\itshape à la Poincar\'e} for the lcs cohomologies.

\begin{prop}\label{prop:poincare}
Let $X$ be a compact differentiable manifold of dimension $2n$ endowed with a locally conformal symplectic form $\Omega$ with Lee form $\vartheta$.
Then, for any weight $k\in\R$, for any degree $h\in\Z$, the symplectic-$\star$-operator induces the isomorphism
$$
\star\colon H^{n-h}_{d_k}(X) \stackrel{\simeq}{\to} H^{n+h}_{\delta_{h+k}}(X) \;.
$$
On the other side, once chosen a compatible triple an almost-K\"ahler structure $(g, J, \Omega)$ on $X$, for any $k\in\R$, $h\in\Z$, the Hodge-$*$-operator induces the isomorphisms
$$
*\colon H^{n-h}_{d_k}(X) \stackrel{\simeq}{\to} H^{n+h}_{d_{-k}}(X)
\qquad \text{ and } \qquad
*\colon H^{n-h}_{\delta_{-k-h}}(X) \stackrel{\simeq}{\to} H^{n+h}_{\delta_{k+h}}(X) \; .
$$
\end{prop}

\begin{proof}
The first statement follows by the formula (\ref{eq:delta-dual-d}):
\begin{eqnarray*}
\delta_{h+k} \star \lfloor_{\wedge^{n-h}X} &=& d^{\star}_{-n-(h+k)+(n+h)}\star \;=\; d^\star_{-k}\star \\[5pt]
&=&
(-1)^{n+h} \star d_{k} \star\star \\[5pt]
&=& (-1)^{n+h} \star d_{k} \lfloor_{\wedge^{n-h}X} \;,
\end{eqnarray*}
and by $\star^2=\mathrm{id}$.

Now let $(g, J, \Omega)$ be a compatible triple. Denoting with
$\mathcal{H}^\bullet_{d_k}(X):=\ker\Delta_{d_k}$ we prove that
$$
*:\mathcal{H}^{n-h}_{d_k}(X) \stackrel{\simeq}{\to} \mathcal{H}^{n+h}_{d_{-k}}(X)\;;
$$
the proof of the other isomorphism is similar.
Let $\alpha\in\mathcal{H}^{n-h}_{d_k}(X)$, namely $d_k\alpha=0$ and
$d_k^*\alpha=0$. Then
$$
d_{-k}*\alpha=(-1)^{n-h+1} **d_{-k}*\alpha= (-1)^{n-h} *d_k^*\alpha
$$
and
$$
d_{-k}^**\alpha=-*d_{k}**\alpha= (-1)^{n-h+1} *d_k\alpha \;.
$$
We have then proved the commutation relation $\Delta_{d_{-k}}*=*\Delta_{d_k}$.
\end{proof}

\begin{thm}\label{thm:poincare-ABC}
Let $X$ be a compact differentiable manifold of dimension $2n$ endowed with a locally conformal symplectic form $\Omega$ with Lee form $\vartheta$.
Let $(g, J, \Omega)$ be an almost-K\"ahler structure on $X$.
Then, for any weight $k\in\R$, for any degree $h\in\Z$, the Hodge-$*$-operator induces the isomorphism
$$
*\colon H^{n-h}_{d_k+\delta_k}(X) \stackrel{\simeq}{\to} H^{n+h}_{\delta_{-k} d_{-k}}(X) \;.
$$
\end{thm}

\begin{proof}
Note that $L^*=*^{-1}L*=\star^{-1}L\star=L^\star=-\Lambda$.
We claim that $\delta_k^* = *\delta_{-k+1}*$. Indeed, by using also $JL=L$ and $J\Lambda=\Lambda$:
\begin{eqnarray*}
\delta_k^*\lfloor_{\wedge^hX} &=& (d_{k-1}\Lambda-\Lambda d_k)^* \;=\; d_k^*L-Ld_{k-1}^* \\[5pt]
&=& - *d_{-k}*L*^{-1}*+ **^{-1}L*d_{-k+1}* \\[5pt]
&=& - *d_{-k}(*^{-1}L*)*+ *(*^{-1}L*)d_{-k+1}* \\[5pt]
&=& *d_{-k}\Lambda*-*\Lambda d_{-k+1}* \\[5pt]
&=& *(d_{-k}\Lambda-\Lambda d_{-k+1})* \\[5pt]
&=& * \delta_{-k+1}* \;.
\end{eqnarray*}
Using this relation and the definitions of the lcs Laplacians, we get that, for any differential form $\alpha$, it holds
$\Delta_{d_k+\delta_k}\alpha=0$ if and only if
$$ d_k\alpha\;=\;0 \;,\qquad \delta_k\alpha\;=\;0\;, \qquad (d_k\delta_{k+1})^*\alpha \;=\;0 \;,$$
equivalently,
$$ d_{-k}^*(*\alpha)\;=\;0 \;,\qquad \delta_{-k+1}^*(*\alpha)\;=\;0\;, \qquad \delta_{-k}d_{-k}(*\alpha)\;=\;0\;,$$
that is, $\Delta_{\delta_{-k} d_{-k}}(*\alpha)=0$. By Proposition \ref{prop:hodge-isom}, we get the proof.
\end{proof}

\subsection{Hard Lefschetz Condition for lcs cohomologies}

As a consequence of the previous relations and their dual we can prove the Hard Lefschetz Condition for the lcs-Bott-Chern and lcs-Aeppli cohomologies
(see \cite[Theorem 3.11, Theorem 3.22]{tseng-yau-1} for the same result in the symplectic setting).

\begin{lem}\label{lem:commutation-relations}
Let $X$ be a manifold endowed with a lcs structure $\Omega$ with Lee form $\theta$.
Then the following commutation relations hold:
$$
Ld_{k}-d_{k+1}L \;=\; 0\;, \qquad
L\delta_{k} - \delta_{k+1} L\;=\; d_k \;,
$$
$$ d_{k-1}\Lambda-\Lambda d_k \;=\; \delta_k\;, \qquad
\delta_{k-1}\Lambda-\Lambda\delta_{k}\;=\;0.$$
\end{lem}
\begin{proof}
The first, \cite[Equation (2.5)]{le-vanzura}, follows by the Leibniz rule and the lcs condition $d_1\Omega=0$. 
The second follows by the first one and by $[L,\Lambda]=H$: indeed,
\begin{eqnarray*}
 L\delta_k-\delta_{k+1}L&=&
Ld_{k-1}\Lambda - L \Lambda d_k-d_k\Lambda L + \Lambda d_{k+1}L\\[5pt]
&=& d_k L \Lambda - L\Lambda d_k -d_k\Lambda L + \Lambda L d_k
\;=\; d_k H - H d_k \\[5pt]
&=& d_k \sum_s (n-s) \pi_{\wedge^s X} - \sum_s (n-s-1) d_k \pi_{\wedge^s X}
\;=\; d_k \;, 
\end{eqnarray*}
where we recall that $H\lfloor_{\wedge^\bullet X}=\sum_s (n-s) \pi_{\wedge^s X}$ where $\pi_{\wedge s X}$ denotes the projection onto the space $\wedge^sX$.
The third and the fourth relations are respectively the definition of $\delta_k$ and the symplectic dual of the first commutation identity above, see \cite[Proposition 2.5]{le-vanzura}.
\end{proof}

\begin{thm}\label{thm:lcs-HLC}
Let $X$ be a compact manifold of dimension $2n$ endowed with a lcs structure $\Omega$ with Lee form $\theta$. Then, for any $h\in\Z$, for any $k\in\R$, the following maps are isomorphisms:
$$
L^h \colon H^{n-h}_{d_k+\delta_k}(X) \stackrel{\simeq}{\to}H^{n+h}_{d_{k+h}+\delta_{k+h}}(X) \;,
$$
$$
L^h \colon H^{n-h}_{\delta_kd_k}(X) \stackrel{\simeq}{\to}H^{n+h}_{\delta_{k+h}d_{k+h}}(X) \;.
$$
\end{thm}
\begin{proof}
We consider the following differential operators
\begin{eqnarray*}
D_{d_k+\delta_k} &:=& d_k^*d_k+\delta_k^*\delta_k+ (\delta_{k+1}d_{k+1})(\delta_{k+1}d_{k+1})^* \;, \\[5pt]
D_{\delta_kd_k} &:=& d_kd_k^*+\delta_{k+1}\delta_{k+1}^*+(\delta_kd_k)^*(\delta_kd_k) \;.
\end{eqnarray*}
Notice that, $\ker D_{d_k+\delta_k}=\ker \Delta_{d_k+\delta_k}$
and $\ker D_{\delta_kd_k}=\ker \Delta_{\delta_kd_k}$.
The advantage of considering these operators is that by the relations proved in Lemma \ref{lem:commutation-relations} one easily gets
$$
LD_{d_k+\delta_k}=D_{d_{k+1}+\delta_{k+1}}L\;,\qquad
LD_{\delta_kd_k}=D_{\delta_{k+1}d_{k+1}}L.
$$
Notice that the operator $L$ does not commute with $\Delta_{d_\bullet+\delta_\bullet}$ and $\Delta_{\delta_\bullet d_\bullet}$.
As a consequence we have that the following maps are isomorphisms
$$
L^h \colon \mathcal{H}^{n-h}_{d_k+\delta_k}(X) \stackrel{\simeq}{\to}\mathcal{H}^{n+h}_{d_{k+h}+\delta_{k+h}}(X) \;,
$$
$$
L^h \colon \mathcal{H}^{n-h}_{\delta_kd_k}(X) \stackrel{\simeq}{\to}\mathcal{H}^{n+h}_{\delta_{k+h}d_{k+h}}(X) \;.
$$
The statement follows by Proposition \ref{prop:hodge-isom}.
\end{proof}

Similarly to \cite[Proposition 1.4]{merkulov}, \cite{guillemin}, \cite[Theorem 5.4]{cavalcanti} stating that the $dd^\Lambda$-Lemma and the Hard Lefschetz Condition are equivalent in the symplectic context, in the lcs setting we have the following result.

\begin{thm}\label{thm:hlc-lcslemma}
Let $X$ be a compact manifold of dimension $2n$ endowed with a lcs structure $\Omega$ with Lee form $\vartheta$. Then, the following conditions are equivalent:
\begin{enumerate}
\item it satisfies the {\em lcs-Hard Lefschetz Condition}, that is, for any $h\in\Z$, for any $k\in\R$, the map
$$ L^h \colon H^{n-h}_{d_k}(X) \to H^{n+h}_{d_{k+h}}(X) $$
is an isomorphism;
\item it satisfies the lcs-Lemma, equivalently, for any $h\in\Z$, for any $k\in\R$, the map
$$ H^h_{d_k+\delta_k}(X) \to H^h_{d_k}(X) $$
is an isomorphism;
\item it is symplectic up to global conformal changes and it satisfies the Hard Lefschetz Condition.
\end{enumerate}
\end{thm}

We will show that {\itshape (1)} gives $[\vartheta]=0$ and then {\itshape (3)}, and that {\itshape (2)} implies {\itshape (1)} because of Theorem \ref{thm:lcs-HLC}; finally, condition {\itshape (3)} is stronger than either {\itshape (1)} and {\itshape (2)} thanks to \cite[Proposition 1.4]{merkulov}, \cite{guillemin}, \cite[Theorem 5.4]{cavalcanti}. For the sake of completeness, we will also give a  proof of the equivalence of {\itshape (1)} and {\itshape (2)}, which may possibly turn useful for weaker statements.
Before proving this we will need few intermediate results.
\begin{prop}\label{prop:hlc-representative}
Let $X$ be a compact manifold endowed with a lcs structure $\Omega$ with Lee form $\vartheta$. Then, the following conditions are equivalent:
\begin{itemize}
\item it satisfies the lcs-Hard Lefschetz Condition;
\item for any $k\in\R$, there exists a $\delta_k$-closed representative in any cohomology class in $H_{d_k}^\bullet(X)$.
\end{itemize}
\end{prop}
\begin{proof}
The proof is an adaptation to the twisted case of the one presented in \cite[Theorem 5.3]{cavalcanti}. We will recall it for completeness.
The "if" implication follows by the following commutative diagram
$$
\xymatrixcolsep{5pc}\xymatrix{
\ker d_k\cap\ker\delta_k\mid_{\Lambda^{n-h}(X)} \ar[r]^{L^h}\ar[d] &
\ker d_{k+h}\cap\ker\delta_{k+h}\mid_{\Lambda^{n+h}(X)}\ar[d]\\
H^{n-h}_{d_k}(X)  \ar[r]^{L^h} &
H^{n+h}_{d_{k+h}}(X).
} 
$$
The left and right vertical arrows are surjective by hypothesis and the top horizontal arrow
is an isomorphism by the commutation relations. Hence the bottom arrow
is surjective.

Suppose now that the lcs-Hard Lefschetz Condition holds. First of all notice that we have the following decomposition
$$
H_{d_k}^{n-h}(X)=\imm L+P^{n-h}
$$
where
$$ P^{n-h}=\left\lbrace [\alpha]\in H^{n-h}_{d_k}(X) \;:\; L^{h+1}[\alpha]=0\right\rbrace$$
and
$$\imm L=\imm \left( L\colon H^{n-h-2}_{d_{k-1}}(X)\to H^{n-h}_{d_k}(X)\right) . $$
Indeed,
let $\alpha\in\wedge^{n-h}X$ be $d_k$-closed. Take $\beta:=L^{h+1}\alpha\in\wedge^{n+h+2}X$: it is a $d_{k+h+1}$-closed form. By the lcs-HLC, there exists $\gamma\in \wedge^{n-h-2}X$ a $d_{k-1}$-closed form such that
$L^{h+2}[\gamma]_{d_{k-1}}=[\beta]_{d_{k+h+1}}$. Therefore,
$$
0=[L^{h+2}\gamma-\beta]_{d_{k+h+1}}=
[L^{h+2}\gamma-L^{h+1}\alpha]_{d_{k+h+1}}=
L^{h+1}[\Omega\wedge\gamma-\alpha]_{d_k},
$$
so $\alpha=L\gamma+(\alpha-\Omega\wedge\gamma)\in\imm L+P^{n-h}$.

Now we prove our thesis by induction on the degree of the form.
If $f$ is a $d_k$-closed smooth function then it is obviously $\delta_k$-closed.
Let $\alpha\in\wedge^1X$ a $d_k$-closed form, then $\delta_k\alpha=
d_{k-1}\Lambda\alpha-\Lambda d_k\alpha=0$. Suppose
that in every class in $H^j_{d_k}(X)$ there exists a $\delta_k$-closed representative for $j<n-h$ and we prove the thesis for degree $n-h$.
Let $\alpha\in\wedge^{n-h}X$ be $d_k$-closed; then by the previous decomposition $\alpha=L\gamma+\tilde\alpha$
with $L^{h+1}[\tilde\alpha]=0$. By induction there exists $\tilde\gamma$ a $\delta_{k-1}$-closed form such that $[\gamma]=[\tilde\gamma]$ and so, if there exists $\psi$ a $\delta_k$-closed form such that $[\psi]=[\tilde\alpha]$, then we conclude the proof.

This last fact follows by the following consideration.
If $\alpha\in\wedge^{n-h}X$ is $d_k$-closed and such that
$L^{h+1}[\alpha]_{d_k}=0$, then there exists a $\delta_k$-closed form $\psi\in\wedge^{n-h}X$ in the same $d_k$-cohomology class. Indeed, since
$L^{h+1}[\alpha]_{d_k}=0$ then $\Omega^{h+1}\wedge\alpha=d_{k+h+1}\tilde\beta$ for some $\tilde\beta\in\wedge^{n+h+1}X$.
Since $L^{h+1}\colon\wedge^{n-h-1}X\to \wedge^{n+h+1}X$ is an isomorphism, there exists $\beta\in\wedge^{n-h-1}X$ such that
$L^{h+1}\beta=\tilde\beta$.
Set $\psi:=\alpha-d_k\beta$. Clearly $d_k\psi=0$ and $[\psi]_{d_k}=[\alpha]_{d_k}$ and
$
L^{h+1}\psi=L^{h+1}\alpha-L^{h+1}d_k\beta=
d_{k+h+1}L^{h+1}\beta-L^{h+1}d_k\beta
=L^{h+1}d_k\beta-L^{h+1}d_k\beta=0$.
Hence, $\psi$ is a primitive $d_k$-closed form so it is $\delta_k$-closed by definition of $\delta_k$.
\end{proof}

\begin{prop}\label{prop:equalities}
Let $X$ be a compact manifold endowed with a lcs structure $\Omega$ with Lee form $\vartheta$. If $X$ satisfy the lcs-Hard Lefschetz condition then the following equalities hold for any $k\in\R$:
\begin{eqnarray*}
\imm\delta_{k+1}\cap\ker d_k&=&\imm d_k\cap\imm\delta_{k+1}, \\[5pt]
\imm d_k\cap\ker \delta_k&=&\imm d_k\cap\imm\delta_{k+1} .
\end{eqnarray*}
\end{prop}
\begin{proof}
We prove the first equality. The second one is similar.

We need to prove that if $\alpha\in\wedge^hX$ is such that
$d_k\delta_{k+1}\alpha=0$ then $\delta_{k+1}\alpha$ is $d_k$-exact. We proceed by induction on the degree of $\alpha$.
If $\alpha$ is a smooth function then clearly $\delta_{k+1}\alpha=0$ is $d_k$-exact. Let $\alpha\in\wedge^1X$ be such that
$d_k\delta_{k+1}\alpha=0$. We have to distinguish two cases.
If $k\neq 0$ then
$\delta_{k+1}\alpha\in H^0_{d_k}(X)=0$ (see {\itshape e.g.} \cite{banyaga}).
Otherwise, if $k=0$, then $\delta_{1}\alpha$ is a $d$-closed $0$-form, so
$\delta_{1}\alpha=c$ constant.
Hence
$$
-\star d_n\star\alpha=\delta_1\alpha=c
$$
and applying $\star$ to the first and the last term in the equalities we get
$-d_n\star\alpha=c\,\mathrm{Vol}=L^nc$, but by hypothesis
$L^n\colon H^0_{d_0}(X)\to H^{2n}_{d_n}(X)$ is an isomorphism and the volume form $\mathrm{Vol}=L^n1$
can not be $d_n$-exact so $0=c=\delta_1\alpha$.

Let now $\alpha\in\wedge^hX$ be such that $d_k\delta_{k+1}\alpha=0$ and take the decomposition
$$\alpha=\sum_rL^r\alpha_r$$
with $\alpha_r$ primitive forms.
It is a straightforward computation to show that
$$
0=d_k\delta_{k+1}\alpha=\sum L^rd_{k-r}\delta_{k+1-r}\alpha_r
$$
with $d_{k-r}\delta_{k+1-r}\alpha_r$ primitive forms; hence every single term is zero, namely $d_{k-r}\delta_{k+1-r}\alpha_r=0$.
When $r>0$, by induction $\delta_{k+1-r}\alpha_r=d_{k-r}\varphi_r$ for some
$\varphi_r\in\wedge^{h-2r-2}X$.
Hence
\begin{eqnarray*}
\delta_{k+1}L^r\alpha_r&=&(L\delta_k-d_k)L^{r-1}\alpha_r\\[5pt]
&=&
L(L\delta_{k-1}-d_{k-1})L^{r-2}\alpha_r-d_kL^{r-1}\alpha_r=\cdots\\[5pt]
&=&L^r\delta_{k-r+1}\alpha_r-rd_kL^{r-1}\alpha_r\\[5pt]
&=&
L^rd_{k-r}\varphi_r-rd_kL^{r-1}\alpha_r\\[5pt]
&=&d_k\left(L^r\varphi_r-rL^{r-1}\alpha_r\right).
\end{eqnarray*}
The last case that we have to consider is when $\alpha\in\wedge^hX$ is
a primitive form.
We define $\beta\in\wedge^{h-1}X$ as
$$
L^{n-h+1}\beta=d_{k+1+n-h}L^{n-h}\alpha.
$$
Notice that $\beta$ is a primitive form, indeed
$$
L^{n-h+2}\beta=Ld_{k+1+n-h}L^{n-h}\alpha=
d_{k+2+n-h}L^{n-h+1}\alpha=0
$$
because $\alpha$ is primitive.
Applying $\Lambda^{n-h+1}$ and by using \cite[Lemma 5.4]{cavalcanti}
we have that there exists a non negative constant $c_{n-h+1,h-1}$ such
that
\begin{eqnarray*}
c_{n-h+1,h-1}\beta&=&\Lambda^{n-h+1}L^{n-h+1}\beta\\[5pt]
&=&
\Lambda^{n-h+1}d_{k+1+n-h}L^{n-h}\alpha\\[5pt]
&=&
\Lambda^{n-h}(d_{k+n-h}\Lambda-\delta_{k+1+n-h})L^{n-h}\alpha=
\cdots\\[5pt]
&=&\left(d_k\Lambda^{n-h+1}-(n-h+1)\delta_{k+1}\Lambda^{n-h}\right)L^{n-h}
\alpha\\[5pt]
&=&-(n-h+1)\delta_{k+1}\Lambda^{n-h}L^{n-h}\\[5pt]
&=&-(n-h+1)\delta_{k+1}c_{n-h,h}\alpha.
\end{eqnarray*}
Applying $L^{n-h+1}$, then there exists $c\neq 0$ such that
$$
cd_{k+1+n-h}L^{n-h}\alpha=cL^{n-h+1}\beta=L^{n-h+1}\delta_{k+1}\alpha.
$$
By the lcs-HLC we have that
$$
L^{n-h+1}\colon H^{h-1}_{d_k}(X)\to H^{2n-h+1}_{d_{k+n-h+1}}(X)
$$
is an isomorphism; since
we have just proven that $L^{n-h+1}[\delta_{k+1}\alpha]=0\in
H^{2n-h+1}_{d_{k+n-h+1}}(X)$ we get that
$$
[\delta_{k+1}\alpha]=0\in H^{h-1}_{d_k}(X)
$$
namely $\delta_{k+1}\alpha$ is $d_k$-exact concluding the proof.
\end{proof}

Now we are ready to proof Theorem \ref{thm:hlc-lcslemma}.
\begin{proof}[Proof of Theorem \ref{thm:hlc-lcslemma}.]
We prove that {\itshape (1)} implies {\itshape (3)}. By hypothesis with $h=n$ and $k=-n$, we have the isomorphism $L^n \colon H^0_{-n}(X) \simeq H^{2n}_{0}(X)$, where clearly $H^{2n}_{0}(X)=H^{2n}_{dR}(X;\R)\simeq\R$. Therefore $H^0_{-n}(X)\neq0$, and this can not happen unless $\vartheta$ is exact
\cite{guedira-lichnerowicz}, \cite[Example 1.6]{haller-rybicki}. To prove the last claim, we can actually argue also as follows. We can choose a generator $f$ for $H^0_{-n}(X)$ having no zero on $M$, since it maps to the volume class by $L^n$. Therefore $df-f\vartheta=0$, that is, $\vartheta=d\lg f$ is exact.

The lcs-Lemma clearly implies the lcs-Hard Lefschetz Condition, thanks to Theorem \ref{thm:lcs-HLC}. Moreover, {\itshape (3)} clearly implies {\itshape (1)}; and {\itshape (3)} implies {\itshape (2)} because of the results in the symplectic case, \cite[Proposition 1.4]{merkulov}, \cite{guillemin}, \cite[Theorem 5.4]{cavalcanti}.

\medskip

For the sake of completeness, now we give also a  proof of the fact that {\itshape (1)} implies {\itshape (2)}; this may possibly be useful if one needs weaker statements.
Suppose that the lcs-Hard Lefschetz Condition holds.
By Proposition \ref{prop:equalities} we are reduced to prove that
$$
\imm d_k\cap \imm \delta_{k+1}=\imm d_k\delta_{k+1}.
$$
Let $\alpha^p=d_k\gamma^{p-1}=\delta_{k+1}\beta^{p+1}\in\wedge^pX$;
we prove that $\alpha^p=d_k\delta_{k+1}\eta$ for some $\eta\in\wedge^pX$.
We prove it by induction on the degree of the form. For $p=0$ and $p=\dim X$,
it is obvious.

For $p=\dim X-1=:2n-1$, we have $d_{k+1}\beta^{2n}=0$ for degree reasons.
Hence, by Proposition \ref{prop:hlc-representative}
there exists $\tilde\beta^{2n}$ such that
$\delta_{k+1}\tilde\beta^{2n}=0$ and
$\beta^{2n}=\tilde\beta^{2n}+d_{k+1}\tau^{2n-1}$ for some $\tau^{2n-1}$.
So,
$$
\alpha^{2n-1}=\delta_{k+1}\beta^{2n}=\delta_{k+1}d_{k+1}\tau^{2n-1}=
d_k\delta_{k+1}(-\tau^{2n-1}).
$$
Now, suppose that the thesis holds for $p=h+2$ and we prove it for $p=h$.
Let $\alpha^h=d_k\gamma^{h-1}=\delta_{k+1}\beta^{h+1}$.
We set $\alpha^{h+2}:=d_{k+1}\beta^{h+1}$ and we get
$$
\delta_{k+1}\alpha^{h+2}=-d_k\delta_{k+1}\beta^{h+1}=0,
$$
namely $\alpha^{h+2}\in \ker\delta_{k+1}\cap \imm d_{k+1}=
\imm d_{k+1}\cap\imm\delta_{k+2}$.
Setting
$\alpha^{h+2}=d_{k+1}\beta^{h+1}=\delta_{k+2}\mu^{h+3}$,
by induction we have
$$\alpha^{h+2}=d_{k+1}\delta_{k+2}\nu^{h+2}.$$
Then
$$
d_{k+1}(\beta^{h+1}-\delta_{k+2}\nu^{h+2})=0
$$
and by Proposition \ref{prop:hlc-representative} there exists
$\tilde\beta^{h+1}\in\wedge^{h+1}X$ such that
$$
\delta_{k+1}\tilde\beta^{h+1}=d_{k+1}\tilde\beta^{h+1}=0,\quad
\beta^{h+1}=\tilde\beta^{h+1}-\delta_{k+2}\nu^{h+2}+d_{k+1}\lambda^h
$$
for some $\lambda^h\in\wedge^hX$.
So,
$$
\alpha^h=\delta_{k+1}\beta^{h+1}=\delta_{k+1}d_{k+1}\lambda^h=
d_k\delta_{k+1}(-\lambda^h)
$$
namely $\alpha^h\in\imm d_k\delta_{k+1}$.
\end{proof}
\begin{rmk}
Notice that if $X$ is a compact lcs manifold with lcs-form $\Omega$
$d_\vartheta$-exact then $\Omega^n$ would be $d_n$-exact and this is not possible
if $X$ satisfies the lcs-Hard Lefschetz condition.
\end{rmk}

\subsection{Further results}

\begin{rmk}[generic vanishing]\label{rmk:vanishing}
Let $X$ be a compact differentiable manifold, endowed with a closed non-exact $1$-form $\vartheta$.
Consider one of the following cases:
\begin{itemize}
\item $X$ is a completely-solvable solvmanifold \cite[Theorem 4.5]{millionshchikov},
\item or, more in general, $X$ is any compact differentiable manifold and $\vartheta$ is non-zero and parallel with respect to the Levi Civita connection associated to some fixed metric \cite[Theorem 4.5]{deleon-lopez-marrero-padron},
\item or, more in general, if $\vartheta$ is nowhere-vanishing \cite[Theorem 1]{pajitnov}, see also \cite[Exercise 4.5.5]{ornea-verbitsky-book}.
\end{itemize}
Then we know that $H^\bullet_{d_k}(X)=0$ except for a finite number of $k\in\R$. It follows that, if $\Omega$ is a lcs structure on $X$ with Lee form $\vartheta$, then also $H^\bullet_{\delta_k}(X)=0$, $H^\bullet_{d_k+\delta_k}(X)=0$, and $H^\bullet_{\delta_kd_k}(X)=0$ except for a finite number of $k\in\R$.
(This follows by symmetries, see Proposition \ref{prop:poincare} and Theorem \ref{thm:poincare-ABC}, and by \cite[Theorem 6.2]{angella-tardini-1}, which can be rewritten in the general context of $\Z$-graded bi-diffential vector spaces.)

In general, there is no generic vanishing, since the Euler characteristic of the Morse-Novikov complex coincides with the Euler characteristic of the manifold, as a consequence of the Atiyah-Singer index theorem, see \cite{BKg}.
\end{rmk}

\section{Twisted cohomologies of solvmanifolds}\label{sec:ex}
Recall that a {\em solvmanifold} $X=\left.\Gamma\middle\backslash G \right.$ (respectively, {\em nilmanifold}) is a compact quotient of a connected simply-connected solvable $G$ (respectively, nilpotent) Lie group by a co-compact discrete subgroup $\Gamma$.
In this section, we provide conditions on $X$ that allow to reduce the computation of the lcs cohomologies at the level of the associated Lie algebra, reducing the problem to a linear problem. We can apply these results on explicit examples in the next section.

\subsection{Hattori theorem for completely-solvable solvmanifolds}
A solvmanifold is said to be {\em completely-solvable} if the eigenvalues of the endomorphisms given by the adjoint representation of the corresponding Lie algebra are all real.
(In particular, note that nilmanifolds are completely-solvable solvmanifolds.)
In this case, the subcomplex of invariant forms inside the complex of forms induces an isomorphisms in de Rham cohomology, in fact, in Morse-Novikov cohomoogy too \cite[Corollary 4.2]{hattori}.
Here, by {\em invariant}, we mean that the lift to the Lie group is invariant with respect to the action of the group on itself given by left-translations.
In particular, it follows that, up to global conformal changes, we can assume that the Lee forms are invariant.

The Hattori result holds in fact for lcs cohomologies.

\begin{lem}\label{lem:inv-lcs-cohom}
 Let $X=\left.\Gamma\middle\backslash G\right.$ be a completely-solvable solvmanifold endowed with an invariant lcs structure. Then the inclusion of invariant forms into the space of forms induces isomorphisms at the level of lcs cohomologies.
\end{lem}

\begin{proof}
Since both the lcs structure $\Omega$ and the Lee form $\vartheta$ are invariant, then the operators $d_k$ and $\delta_k$ preserve the space of invariant forms.
Left-translations induce maps
$$ H^{\bullet}_{\sharp_k}(\mathfrak{g}^*) \to H^{\bullet}_{\sharp_k}(X) \;,$$
varying $\sharp_k\in \{d_k, \delta_k, d_k+\delta_k, \delta_kd_k\}$, for every $k\in\Z$;
where $H^{\bullet}_{\sharp_k}(\mathfrak{g}^*)$ denotes the cohomology of the corresponding bi-differential complex at the level of the Lie algebra $\mathfrak{g}$ of $G$, equivalently, of the space of invariant forms.
The above maps are injective, as a consequence of elliptic Hodge theory in Proposition \ref{prop:hodge-isom}, with respect to an invariant metric compatible with the lcs structure: see the argument in \cite[Lemma 9]{console-fino}.
In fact, by \cite{hattori}, under the assumption that $G$ is completely-solvable, the map
$$ H^{\bullet}_{d_k}(\mathfrak{g}^*) \to H^{\bullet}_{d_k}(X) $$
is an isomorphism. Note that, the lcs structure being invariant, the Poincar\'e isomorphism in Proposition \ref{prop:poincare} is compatible with the inclusion of invariant forms. Then also the map
$$ H^{\bullet}_{\delta_k}(\mathfrak{g}^*) \to H^{\bullet}_{\delta_k}(X) $$
is an isomorphism.
Finally, the fact that the maps
$$ H^{\bullet}_{d_k+\delta_k}(\mathfrak{g}^*) \to H^{\bullet}_{d_k+\delta_k}(X)
\quad\text{ and }\quad
H^{\bullet}_{\delta_kd_k}(\mathfrak{g}^*) \to H^{\bullet}_{\delta_kd_k}(X) $$
are isomorphisms can be deduced from the above isomorphisms for $H_{d_k}$ and $H_{\delta_k}$, see the general argument in \cite[Theorem 2.7]{angella-1} as adapted to the $\Z$-graded context in \cite[Corollary 1.3]{angella-kasuya-3}, and by Poincar\'e duality in Theorem \ref{thm:poincare-ABC}.
\end{proof}

\subsection{Mostow condition for solvmanifolds}
Consider a solvmanifold $X=\left.\Gamma \middle\backslash G \right.$, and let $\mathfrak{g}$ be its associated Lie algebra.
The isomorphism $H_{d}^\bullet(\mathfrak{g}^*)\stackrel{\simeq}{\to}H_{d}^\bullet(\Gamma\backslash G)$ holds also under the {\em Mostow condition} that $\mathrm{Ad}\,(\Gamma)$ and $\mathrm{Ad}\,(G)$ have the same Zariski closure in $\mathrm{GL}(\mathfrak{g})$ (where we understand by $\mathrm{GL}(\mathfrak{g})$ the group consisting solely of the linear isomorphisms of $\mathfrak{g}$) \cite[Corollary 8.1]{mostow}. In fact Mostow considers cohomology $H^\bullet(\Gamma\backslash G; \rho)$ with $\rho$ a representation of $G$ in a vector space $F$, assuming that $\Gamma$ is {\em $\rho$-ample} \cite[Section 6]{mostow} (say, $\rho$ is {\em $\Gamma$-admissible} in the notation of \cite[Definition 7.24]{raghunathan}.)
This means that $\rho\oplus\mathrm{Ad}$, as a representation of $G$ in $F\oplus \mathfrak{g}$, satisfy that $\overline{(\rho\oplus\mathrm{Ad})\,(\Gamma)}=\overline{(\rho\oplus\mathrm{Ad})\,(G)}$, where the closure is with respect to the Zariski topology. In this case, one has that the restriction morphism $H^{\bullet}(\mathfrak{g};\rho)\simeq H^{\bullet}(G;\rho) \to H^{\bullet}(\Gamma;\rho)$ is an isomorphism, \cite[Theorem 8.1]{mostow}, see also \cite[Theorem 7.26]{raghunathan}.
In particular the assumption holds: when $\rho$ is a unipotent representation of a nilpotent Lie group $G$; when $G$ satisfies the Mostow condition $\overline{\mathrm{Ad}\,(\Gamma)} = \overline{\mathrm{Ad}\,(G)}$ and $\rho$ is trivial; see \cite[Theorem 8.2]{mostow}.
As explicit application, we write down as the result applies to Morse-Novikov cohomologies.

\begin{prop}\label{prop:mostow}
Consider a solvmanifold satisfying the Mostow condition.
Then the inclusion of invariant forms into the space of forms induces isomorphisms at the level of Morse-Novikov cohomology with respect to any invariant Lee form.
Moreover, if $X$ is endowed with an invariant lcs structure, then the same holds true at the level of lcs cohomologies.
\end{prop}

\begin{proof}
Let $X=\left.\Gamma\middle\backslash G\right.$ be a solvmanifold such that the Mostow condition holds. Denote by $\mathfrak{g}$ its Lie algebra. Let $\vartheta$ be an invariant closed $1$-form. In the case $\vartheta$ is exact, we are reduced to the Mostow theorem \cite[Corollary 8.1]{mostow}; hence, assume $\vartheta$ is not exact.
We want to prove that the natural map $H^{\bullet}_{\vartheta}(\mathfrak{g})\to H^{\bullet}_{\vartheta}(X)$ is an isomorphism.
Let $\pi^*\vartheta=:\tilde\vartheta$ be the $d$-exact invariant $1$-form on $G$ that lifts $\vartheta$, where $\pi\colon G \to X$. Consider
$$ \rho \colon G \times \mathbb{R}\to\mathbb{R} , \qquad \rho(g)(r):=\exp\left(\int_e^g \tilde\vartheta\right) \cdot r, $$
where $\int_e^g$ is the integral over any path in $G$ connecting the identity $e$ to the element $g$; recall that $G$ is simply-connected. Since $\tilde\vartheta$ is invariant under left-translations, then $\rho$ is a representation of $G$ in $\mathbb{R}$. When restrited to $\Gamma=\pi_1(X)$, which is isomorphic to the deck group of the cover $\pi\colon G \to X$, it is equivalent to the representation
$$ \chi\colon \pi_1(X) \times \mathbb{R}\to\mathbb{R}, \qquad \chi([\gamma])(r):=\exp\left(\int_\gamma\vartheta\right)\cdot r .$$
Therefore
$$ H_{\vartheta}^\bullet(X) \simeq H^\bullet(X;L_\chi) \simeq H^\bullet(X;L_\rho) \simeq H^\bullet(\Gamma;\rho) , $$
where $L_{\rho}$ denotes the flat real line bundle associated to the representation $\rho$, and where the last isomorphism follows from \cite[Lemma 7.4]{raghunathan} since $G$ is contractible. Then, we are reduced to prove that $\chi$ is {\em $\Gamma$-supported}, that is $\overline{\chi(\Gamma)}=\overline{\chi(G)}$, where overline denotes the Zariski closure in $\mathrm{Aut}_{\mathbb{R}}(\mathbb{R})=\mathbb{R}^\times$: the statements then follows by \cite[Theorem 8.1]{mostow}. Here, the topology in $\mathbb{R}^\times$ is the one induced by $\mathbb{R}^2$ where $\mathbb{R}^\times$ is seen as a Zariski closed set. Note that $\chi(\Gamma)$ is identified with a subgroup of the torsion-free group $(\mathbb{R}_{>0},\cdot)$, hence it is either trivial or infinite. However, if it were trivial, the periods $\int_{\gamma} \vartheta$ would vanish for all $\gamma \in H_1(X)$, meaning that $\vartheta$ is exact, which is not the case.  So $\chi(\Gamma)$ is infinite. Then $\overline{\chi(\Gamma)}=\mathbb{R}^\times$, whence also $\overline{\chi(G)}=\mathbb{R}^\times$.

The last statement follows as in Lemma \ref{lem:inv-lcs-cohom}.
\end{proof}

\section{Examples}\label{sec:ex}

In this section, we discuss some examples.

\subsection{Kodaira-Thurston surface}
As an example, we consider the Kodaira-Thurston surface $X$ \cite{kodaira-surfaces-1, thurston}.
Recall that a (primary) Kodaira surface is a compact complex surface with Kodaira dimension $0$, first Betti number odd and trivial canonical bundle.
It admits both complex and symplectic structures, but it has no K\"ahler structure \cite{thurston}.
It is a homogeneous manifold of nilpotent Lie group, \cite[Theorem 1]{hasegawa-jsg}. More precisely, the connected simply-connected covering Lie group is the product of the real three dimension Heisenberg group and the real $1$-dimensional torus. Denote its Lie algebra by $\mathfrak{r}\mathfrak{h}_3=\mathfrak{g}_{3.1}\oplus\mathfrak{g}_{1}$.

We choose a co-frame of invariant $1$-forms $\{e^1, e^2, e^3, e^4\}$ with structure equations
$$
de^1 \;=\; 0 \;, \qquad
de^2 \;=\; 0 \;, \qquad
de^3 \;=\; e^1 \wedge e^2 \;, \qquad
de^4 \;=\; 0 \;.
$$
The almost-symplectic form
\begin{equation}\label{eq:lcs-KT}
\Omega \;:=\; e^1\wedge e^2+ e^3\wedge e^4
\end{equation}
is a locally conformally symplectic structure with Lee form
$$ \vartheta \;:=\; e^4 \;. $$
In fact, $\Omega=d_{\vartheta}(e^3)$ is $d_{\vartheta}$-exact.
Up to equivalence, this is the only lcs structure on the Lie algebra $\mathfrak{rh}_3$, see \cite{angella-bazzoni-parton}.
It admits a compatible complex structure $J$; more precisely, consider the almost-K\"ahler structure
$$ J e^1 \;:=\; e^2 \;, \quad J e^3 \;:=\; e^4 \qquad
\text{ and }
\qquad g \;=\; \sum_{j=1}^{4} e^j\odot e^j \;. $$

Thanks to Lemma \ref{lem:inv-lcs-cohom}, we can compute the lcs cohomologies of the Kodaira-Thurston surface. (As a matter of notation, we have shortened, {\itshape e.g.} $e^{124}:=e^1\wedge e^2\wedge e^4$. Computations have been performed with the help of Sage \cite{sage}.)
\begin{prop}
 The lcs cohomologies of the Kodaira-Thurston surface endowed with the lcs structure in \eqref{eq:lcs-KT} are summarized in Table \ref{table:lcs-cohom-rh-3}.
\end{prop}

\begin{center}
\begin{table}[ht]
 \centering
{\resizebox{\textwidth}{!}{
\begin{tabular}{>{$\mathbf\bgroup}c<{\mathbf\egroup$} >{$\mathbf\bgroup}c<{\mathbf\egroup$} || >{$}c<{$} | >{$}c<{$} | >{$}c<{$} | >{$}c<{$} | >{$}c<{$} ||}
\toprule
\multicolumn{2}{c||}{$\mathbf{H^h_{\sharp_{k}}}$} & \mathbf{k=-2} & \mathbf{k=-1} & \mathbf{k=0} & \mathbf{k=1} & \mathbf{k=2} \\
\toprule
\multirow{4}{*}{$\mathbf{h=0}$}
& H^0_{d_k} & - & - & \langle 1 \rangle & - & - \\
& H^0_{\delta_k} & \langle 1 \rangle & - & - & - & - \\
& H^0_{d_k+\delta_k} & - & - & \langle 1 \rangle & - & - \\
& H^0_{\delta_kd_k} & \langle 1 \rangle & - & - & - & - \\
\midrule[0.02em]
\multirow{4}{*}{$\mathbf{h=1}$}
& H^1_{d_k} & - & - & \langle e^1, e^2, e^4 \rangle & - & - \\
& H^1_{\delta_k} & - & \langle e^1, e^2, e^3 \rangle & - & - & - \\
& H^1_{d_k+\delta_k} & \langle e^4 \rangle & - & \langle e^1, e^2, e^4 \rangle & - & - \\
& H^1_{\delta_kd_k} & - & \langle e^1, e^2, e^3 \rangle & - & \langle e^3 \rangle & - \\
\midrule[0.02em]
\multirow{4}{*}{$\mathbf{h=2}$}
& H^2_{d_k} & - & - & \langle e^{13}, e^{14}, e^{23}, e^{24} \rangle & - & - \\
& H^2_{\delta_k} & - & - & \langle e^{13}, e^{14}, e^{23}, e^{24} \rangle & -
& - \\
& H^2_{d_k+\delta_k} & - & \langle e^{14}, e^{24}, e^{12}-e^{34}\rangle &
\langle e^{13}, e^{14}, e^{23}, e^{24} \rangle & \langle e^{12}+e^{34}\rangle & - \\
& H^2_{\delta_kd_k} & - & \langle e^{12}+e^{34}\rangle & \langle e^{13}, e^{14}, e^{23}, e^{24} \rangle & \langle e^{13}, e^{23},e^{12}-e^{34} \rangle & - \\
\midrule[0.02em]
\multirow{4}{*}{$\mathbf{h=3}$}
& H^3_{d_k} & - & - & \langle e^{123}, e^{134}, e^{234} \rangle & - & - \\
& H^3_{\delta_k} & - & - & - & \langle e^{124}, e^{134}, e^{234} \rangle & - \\
& H^3_{d_k+\delta_k} & - & \langle e^{124} \rangle & - & \langle e^{124}, e^{134}, e^{234} \rangle & - \\
& H^3_{\delta_kd_k} & - & - & \langle e^{123},e^{134}, e^{234} \rangle & - & \langle e^{123} \rangle \\
\midrule[0.02em]
\multirow{4}{*}{$\mathbf{h=4}$}
& H^4_{d_k} & - & - & \langle e^{1234} \rangle & - & - \\
& H^4_{\delta_k} & - & - & - & - & \langle e^{1234} \rangle \\
& H^4_{d_k+\delta_k} & - & - & - & - & \langle e^{1234} \rangle \\
& H^4_{\delta_kd_k} & - & - & \langle e^{1234} \rangle & - & - \\
\midrule[0.02em]
\bottomrule
\end{tabular}
}}
\caption{The lcs cohomologies of the Kodaira-Thurston surface. (Just non-trivial cohomology groups are reported.)}
\label{table:lcs-cohom-rh-3}
\end{table}
\end{center}

\begin{center}
\begin{table}[ht]
 \centering
\begin{tabular}{>{$\mathbf\bgroup}c<{\mathbf\egroup$} >{$\mathbf\bgroup}c<{\mathbf\egroup$} || >{$}c<{$} | >{$}c<{$} | >{$}c<{$} | >{$}c<{$} | >{$}c<{$} ||}
\toprule
\multicolumn{2}{c||}{$\dim\mathbf{H^h_{\sharp_{k}}}$} & \mathbf{k=-2} & \mathbf{k=-1} & \mathbf{k=0} & \mathbf{k=1} & \mathbf{k=2} \\
\toprule
\multirow{4}{*}{$\mathbf{h=0}$} & H^0_{d_k} & - & - & 1 & - & - \\
& H^0_{\delta_k} & 1 & - & - & - & - \\
& H^0_{d_k+\delta_k} & - & - & 1 & - & - \\
& H^0_{\delta_kd_k} & 1 & - & - & - & - \\
\midrule[0.02em]
\multirow{4}{*}{$\mathbf{h=1}$} & H^1_{d_k} & - & - & 3 & - & - \\
& H^1_{\delta_k} & - & 3 & - & - & - \\
& H^1_{d_k+\delta_k} & 1 & - & 3 & - & - \\
& H^1_{\delta_kd_k} & - & 3 & - & 1 & - \\
\midrule[0.02em]
\multirow{4}{*}{$\mathbf{h=2}$} & H^2_{d_k} & - & - & 4 & - & - \\
& H^2_{\delta_k} & - & - & 4 & - & - \\
& H^2_{d_k+\delta_k} & - & 3 & 4 & 1 & - \\
& H^2_{\delta_kd_k} & - & 1 & 4 & 3 & - \\
\midrule[0.02em]
\multirow{4}{*}{$\mathbf{h=3}$} & H^3_{d_k} & - & - & 3 & - & - \\
& H^3_{\delta_k} & - & - & - & 3 & - \\
& H^3_{d_k+\delta_k} & - & 1 & - & 3 & - \\
& H^3_{\delta_kd_k} & - & - & 3 & - & 1 \\
\midrule[0.02em]
\multirow{4}{*}{$\mathbf{h=4}$} & H^4_{d_k} & - & - & 1 & - & - \\
& H^4_{\delta_k} & - & - & - & - & 1 \\
& H^4_{d_k+\delta_k} & - & - & - & - & 1 \\
& H^4_{\delta_kd_k} & - & - & 1 & - & - \\
\midrule[0.02em]
\bottomrule
\end{tabular}
\caption{Summary of the dimensions of the lcs cohomologies of the Kodaira-Thurston surface. (Just non-trivial cohomology groups are reported.)}
\label{table:dim-lcs-cohom-rh-3}
\end{table}
\end{center}

\subsection{Lie algebra \texorpdfstring{$\mathfrak{d}_{4}$}{d4}}
As a further example, we study the Lie algebra $\mathfrak{d}_4=\mathfrak{g}_{4.8}^{-1}$, that is, the Lie algebra associated to the Inoue surface of type $S^+$ \cite{inoue}.
It is completely-solvable. It has structure equations $(14,-24,12,0)$, namely, there exists a basis $\{e_1,e_2,e_3,e_4\}$ such that the dual basis $\{e^1,e^2,e^3,e^4\}$ satisfies
$$ d e^1 \;=\; e^1\wedge e^4\;, \quad de^2 \;=\; - e^2\wedge e^4\;, \quad de^3 \;=\; -e^1\wedge e^2\;, \quad de^4\;=\;0 \;.$$

Consider the lcs structure
$$ \Omega \;:=\; e^1\wedge e^2 + e^3 \wedge e^4 \qquad \text{ with Lee form } \vartheta \;:=\;-e^4 \;.$$
In fact, $\Omega=d_\vartheta(-e^3)$.

The results for the lcs cohomologies are summarized in Tables \ref{table:lcs-cohom-d4} and \ref{table:dim-lcs-cohom-d4}.

\begin{center}
\begin{table}[ht]
 \centering
{\resizebox{\textwidth}{!}{
\begin{tabular}{>{$\mathbf\bgroup}c<{\mathbf\egroup$} >{$\mathbf\bgroup}c<{\mathbf\egroup$} || >{$}c<{$} | >{$}c<{$} | >{$}c<{$} | >{$}c<{$} | >{$}c<{$} ||}
\toprule
\multicolumn{2}{c||}{$\mathbf{H^h_{\sharp_{k}}}$} & \mathbf{k=-2} & \mathbf{k=-1} & \mathbf{k=0} & \mathbf{k=1} & \mathbf{k=2} \\
\toprule
\multirow{4}{*}{$\mathbf{h=0}$}
& H^0_{d_k} & - & - & \langle 1 \rangle & - & - \\
& H^0_{\delta_k} & \langle 1 \rangle & - & - & - & - \\
& H^0_{d_k+\delta_k} & - & - & \langle 1 \rangle & - &  - \\
& H^0_{\delta_kd_k} & \langle 1 \rangle & - & - & - & - \\
\midrule[0.02em]
\multirow{4}{*}{$\mathbf{h=1}$}
& H^1_{d_k} & - & \langle e^2 \rangle & \langle e^4 \rangle & \langle e^1 \rangle & - \\
& H^1_{\delta_k} & \langle e^2 \rangle & \langle e^3 \rangle & \langle e^1 \rangle & - & -  \\
& H^1_{d_k+\delta_k} & \langle e^4 \rangle & \langle e^2\rangle & \langle e^4 \rangle & \langle e^1 \rangle &  - \\
& H^1_{\delta_kd_k} & \langle e^2 \rangle & \langle e^3 \rangle & \langle e^1 \rangle & \langle e^3 \rangle & - \\
\midrule[0.02em]
\multirow{4}{*}{$\mathbf{h=2}$} & H^2_{d_k} & - & \langle e^{23}, e^{24} \rangle & - & \langle e^{13}, e^{14} \rangle & - \\
& H^2_{\delta_k} & - & \langle e^{23}, e^{24} \rangle & - & \langle e^{13}, e^{14} \rangle & - \\
& H^2_{d_k+\delta_k} & \langle e^{24} \rangle & \langle e^{23}, e^{24},e^{12}-e^{34} \rangle & \langle e^{14} \rangle & \langle e^{13}, e^{14},e^{12}+e^{34} \rangle & - \\
& H^2_{\delta_kd_k} & - & \langle e^{23}, e^{24},e^{12}+e^{34} \rangle & \langle e^{23} \rangle  & \langle e^{13}, e^{14},e^{12}-e^{34} \rangle & \langle e^{13} \rangle  \\
\midrule[0.02em]
\multirow{4}{*}{$\mathbf{h=3}$}
& H^3_{d_k} & - & \langle e^{234} \rangle & \langle e^{123} \rangle & \langle e^{134} \rangle & - \\
& H^3_{\delta_k} & - & - & \langle e^{234} \rangle & \langle e^{124} \rangle & \langle e^{134} \rangle \\
& H^3_{d_k+\delta_k} & - & \langle e^{124} \rangle & \langle e^{234} \rangle & \langle e^{124} \rangle & \langle e^{134} \rangle \\
& H^3_{\delta_kd_k} & - & \langle e^{234} \rangle & \langle e^{123} \rangle & \langle e^{134} \rangle & \langle e^{123} \rangle \\
\midrule[0.02em]
\multirow{4}{*}{$\mathbf{h=4}$}
& H^4_{d_k} & - & - & \langle e^{1234} \rangle & - & - \\
& H^4_{\delta_k} & - & - & - & - & \langle e^{1234} \rangle \\
& H^4_{d_k+\delta_k} & - & - & - & - & \langle e^{1234} \rangle \\
& H^4_{\delta_kd_k} & - & - & \langle e^{1234} \rangle & - & - \\
\midrule[0.02em]
\bottomrule
\end{tabular}
}}
\caption{The lcs cohomologies of the Inoue surface of type $S^+$. (Just non-trivial cohomology groups are reported.)}
\label{table:lcs-cohom-d4}
\end{table}
\end{center}

\begin{center}
\begin{table}[ht]
 \centering
\begin{tabular}{>{$\mathbf\bgroup}c<{\mathbf\egroup$} >{$\mathbf\bgroup}c<{\mathbf\egroup$} || >{$}c<{$} | >{$}c<{$} | >{$}c<{$} | >{$}c<{$} | >{$}c<{$} ||}
\toprule
\multicolumn{2}{c||}{$\dim\mathbf{H^h_{\sharp_{k}}}$} & \mathbf{k=-2} & \mathbf{k=-1} & \mathbf{k=0} & \mathbf{k=1} & \mathbf{k=2}  \\
\toprule
\multirow{4}{*}{$\mathbf{h=0}$}
& H^0_{d_k} & - & - & 1 & - & - \\
& H^0_{\delta_k} & 1 & - & - & - & - \\
& H^0_{d_k+\delta_k} & - & - & 1 & - & - \\
& H^0_{\delta_kd_k} & 1 & - & - & - & - \\
\midrule[0.02em]
\multirow{4}{*}{$\mathbf{h=1}$}
& H^1_{d_k} & - & 1 & 1 & 1 & - \\
& H^1_{\delta_k} & 1 & 1 & 1 & - & - \\
& H^1_{d_k+\delta_k} & 1 & 1 & 1 & 1 & - \\
& H^1_{\delta_kd_k} & 1 & 1 & 1 & 1 & - \\
\midrule[0.02em]
\multirow{4}{*}{$\mathbf{h=2}$}
& H^2_{d_k} & - & 2 & - & 2 & - \\
& H^2_{\delta_k} & - & 2 & - & 2 &  - \\
& H^2_{d_k+\delta_k} & 1 & 3 & 1 & 3 & - \\
& H^2_{\delta_kd_k} & - & 3 &  1 & 3 & 1 \\
\midrule[0.02em]
\multirow{4}{*}{$\mathbf{h=3}$}
& H^3_{d_k} & - & 1 & 1 & 1 & - \\
& H^3_{\delta_k} & - & - & 1 & 1 & 1 \\
& H^3_{d_k+\delta_k} & - & 1 & 1 & 1 & 1 \\
& H^3_{\delta_kd_k} & - & 1 & 1 & 1 & 1 \\
\midrule[0.02em]
\multirow{4}{*}{$\mathbf{h=4}$} & H^4_{d_k} & - & - & 1 & - & - \\
& H^4_{\delta_k} & - & - & - & - & 1 \\
& H^4_{d_k+\delta_k} & - & - & - & - & 1 \\
& H^4_{\delta_kd_k} & - & - & 1 & - & - \\
\midrule[0.02em]
\bottomrule
\end{tabular}
\caption{Summary of the dimensions of the lcs cohomologies of the Inoue surface of type $S^+$. (Just non-trivial cohomology groups are reported.)}
\label{table:dim-lcs-cohom-d4}
\end{table}
\end{center}

\subsection{Inoue surfaces $S^0$}
We prove here that the Inoue surfaces of type $S^0$ satisfy the Mostow condition, and then Proposition \ref{prop:mostow} applies for them. This is in accord with the results in \cite{otiman-2} by the second-named author. Since the Inoue surfaces of type $S^{\pm}$ are completely-solvable then the Hattori theorem \cite[Corollary 4.2]{hattori} applies.

\begin{prop}\label{prop:mostow-inoue}
Inoue surfaces of type $S^0$ satisfy the Mostow condition.
\end{prop}

\begin{proof}
Let $S^0:=S^0_A$ be the Inoue surface associated to the matrix $A\in\mathrm{SL}(3;\mathbb{Z})$ with eigenvalues $\alpha>1$, $\beta$, $\bar\beta$, where $\beta\not\in\mathbb{R}$. Recall that $\alpha\not\in\Q$, otherwise $|\alpha|=1$ since $\det A=1$.

We first claim that Gorbatsevich criterion \cite[Theorem 4]{gorbatsevich} for Inoue surfaces reads as follows: $S^0$ satisfies the Mostow condition if and only if there exist $q\in\Q$ such that
\begin{equation}\label{eq:gorb-inoue}
\beta = \alpha^{-1/2}\exp(\sqrt{-1}q\pi) .
\end{equation}

Recall that Gorbatsevich criterion applies to quotients of almost-Abelian Lie groups $G=\R\ltimes_{\varphi}\R^{n}$ by lattices $\Gamma=\Z\ltimes_{\varphi}\Z^n$, where $\varphi(t)=\exp(tZ)$. Let $t_0$ be a generator of $\Z$ in $\Gamma$. Then \cite[Theorem 4]{gorbatsevich} states that $\left.\Gamma\middle\backslash G\right.$ satisfies the Mostow condition if and only if $\sqrt{-1}\pi$ is not a linear combination with rational coefficients of the elements in the spectrum of $t_0 Z$.

In our case, we look at $S^0 = \left. \Z\ltimes \Z^3 \middle \backslash \R\ltimes(\C\times\R) \right.$, where the action is
$$
\R \times (\C \times \R) \ni (t,(z,r)) \mapsto (\beta^t\cdot z, \alpha^t\cdot r) \in \C\times\R .
$$

Here $\Z^3$ is the lattice generated by the eigenvectors of $A$. 
Then we have
$$ \varphi(1) = \left(
\begin{matrix}
\Re \beta & \Im \beta & 0 \\
-\Im \beta & \Re \beta & 0 \\
0 & 0 & \alpha
\end{matrix}
\right) .$$
Since $\det A = \alpha|\beta|^2=1$, we have that
$$ \beta=\frac{1}{\sqrt{\alpha}}\exp(\sqrt{-1}s) $$
for some $s\in\R$.
Then we can take
$$
Z = \left(
\begin{matrix}
-\frac{\lg\alpha}{2} & s & 0 \\
-s & -\frac{\lg\alpha}{2} & 0 \\
0 & 0 & \lg \alpha
\end{matrix}
\right) .
$$
The eigenvalues of $Z$ are:
$$ \lg\alpha , \qquad -\frac{\lg\alpha}{2}+\sqrt{-1}s , \qquad -\frac{\lg\alpha}{2}-\sqrt{-1}s . $$
Then, $\sqrt{-1}\pi$ is a linear combination with rational coefficients of the elements in the spectrum of $Z$ if and only if there exist $x,y,z\in\Q$ such that
$$ x-\frac{1}{2}y-\frac{1}{2}z=0 \qquad \text{ and } \qquad (y-z) s = \pi , $$
namely, if and only if there exists $q\in\Q$ such that
$$ s = q \pi , $$
proving the claim.

We now prove that \eqref{eq:gorb-inoue} does not hold, for any $q\in\Q$. On the contrary, assume that $m\in\Z$ and $n\in\Z\setminus\{0\}$ satisfy
$$ \beta = \alpha^{-1/2}\exp\left(\sqrt{-1}\frac{m}{n}\pi\right) .$$
In particular, $\beta^{2n}=\bar\beta^{2n}=\alpha^{-n}$.
By considering the characteristic polynomial of $A$, that is $x^3-ax^2+bx-1$, where $a=\alpha+\beta+\bar\beta\in\Z$ and $b=\alpha\beta+\alpha\bar\beta+|\beta|^2\in\Z$, we get that $\beta^3=a\beta^2-b\beta+1$. By induction, for any $k\in\N$, $k\geq3$:
$$ \beta^k=x_k\beta^2+y_k\beta+z_k $$
where
$$ x_{k+1}=ax_k+y_k, \qquad y_{k+1}=z_k-x_kb, \qquad z_{k+1}=x_k , $$
with the base condition:
$$ x_{3}=a, \qquad y_{3}=-b, \qquad z_{3}=1 .$$
Using that $\beta\neq\bar\beta$ , equation $\beta^{2n}=\bar\beta^{2n}$ now reads as
$$ x_{2n}(\beta+\bar\beta)+y_{2n} =0 .$$
Using that $a-\alpha=\beta+\bar\beta$ we get
$$ ax_{2n}+y_{2n}=x_{2n}\alpha ,$$
where the left-hand side is $x_{2n+1}\in\Z$ and the right-hand side is the product of $x_{2n}\in\Z$ and of $\alpha\in\R\setminus\Q$. Hence we get that $x_{2n}=y_{2n}=0$, and then $\beta^{2n}=\bar\beta^{2n}=\alpha^{-n}=z_{2n}\in\Z$.

Consider now the polynomial $x^{2n}-z_{2n}\in\Z[x]$, and its division by the characteristic polynomial of $A$ in $\Q[x]$:
$$ x^{2n}-z_{2n} = Q(x) \cdot (x^3-ax^2+bx-1)+R(x) , $$
where $Q(x)\in\Q[x]$ and $R(x)\in\Q_{\leq2}[x]$.
If $R(x)$ had positive degree, then $R(\beta)=R(\bar\beta)=0$ would imply $\beta+\bar\beta\in\Q$, which is not true since $\beta+\bar\beta=a-\alpha$ with $\alpha$ irrational. Then $R(x)=0$. It follows that $\alpha^{2n}=z_{2n}$, too. But this is a contradiction with $\alpha^{-n}=z_{2n}$, since $\alpha>1$.
\end{proof}

\subsection{Oeljeklaus-Toma manifolds with precisely one complex place}
We now extend the above results to Oeljeklaus-Toma manifolds \cite{oeljeklaus-toma} with precisely one complex place and $s$ real place. Note that this is the case when the existence of lcK metrics is known, \cite[Proposition 2.9]{oeljeklaus-toma}, see also \cite[Theorem 3.1]{vuletescu}. In case $s=1$, we recover any Inoue surfaces $S^0_A$ of type $S^0$ by taking $K=\Q(\alpha)$ and $U=\mathcal{O}_K^{*,+}$ generated by $\alpha$, the real eigenvalue of the matrix $A\in\mathrm{SL}(3;\Z)$.

\medskip

We briefly recall their construction (see \cite{oeljeklaus-toma}) and their structure as solvmanifolds (see \cite[Section 6]{kasuya-blms}).

Let $K$ be an algebraic number field. Consider the $n=s+2t$ embeddings of the field $K$ in $\C$: more precisely, the $s$ real embeddings $\sigma_1,\ldots,\sigma_s\colon K\to \R$, and the $2t$ complex embeddings $\sigma_{s+1},\ldots,\sigma_{s+t},\sigma_{s+t+1}=\overline\sigma_{s+1},\ldots,\sigma_{s+2t}=\overline\sigma_{s+t} \colon K\to\C$.
Denote by $\mathcal O_K$ the ring of algebraic integers of $K$, and by $\mathcal O_K^{*,+}$ the group of totally positive units. Let $\mathbb H:=\left\{z\in\C \;:\; \Im z>0\right\}$ denote the upper half-plane. On $\mathbb H^s\times\C^t$, consider the action $\mathcal{O}_K \circlearrowleft \mathbb H^s\times\C^t$ given by translations,
$$ T_a(w_1,\ldots,w_s,z_{s+1},\ldots,z_{s+t}) := (w_1+\sigma_1(a),\ldots,z_{s+t}+\sigma_{s+t}(a)) , $$
and the action $\mathcal{O}_K^{*,+} \circlearrowleft \mathbb H^s\times\C^t$ given by rotations,
$$ R_u(w_1,\ldots,w_s,z_{s+1},\ldots,z_{s+t}) := (w_1\cdot \sigma_1(u),\ldots,z_{s+t}\cdot \sigma_{s+t}(u)) . $$
Oeljeklaus and Toma proved in \cite[page 162]{oeljeklaus-toma} that there always exists a subgroup $U\subset\mathcal{O}_K^{*,+}$ such that the action $\mathcal{O}_K\rtimes U \circlearrowleft \mathbb{H}^s\times\C^t$ is fixed-point-free, properly discontinuous, and co-compact. The {\em Oeljeklaus-Toma manifold} (say, OT manifold) associated to the algebraic number field $K$ and to the admissible subgroup $U$ of $\mathcal O_K^{*,+}$ is
$$ X(K,U) \;:=\; \left. \mathbb H^s\times\C^t \middle\slash \mathcal O_K\rtimes U \right.$$
Moreover, $X(K,U)$ is called {\em of simple type} if there is no intermediate extension $\Q \subset K' \subset K$ such that $U$ is compatible with $K'$, too.

Oeljeklaus-Toma manifolds are in fact solvmanifolds, see \cite[Section 6]{kasuya-blms}. More precisely, consider the map
$$ \ell \colon \mathcal{O}_K^{*,+} \to \R^{s+t}, $$
$$ \ell(u) = \left( \lg \sigma_1(u), \ldots, \lg \sigma_s(u), 2\lg |\sigma_{s+1}(u)|, \ldots, 2\lg |\sigma_{s+t}(u)| \right) .$$
The rank $s$ subgroup $U$ is such that its projection on the first $s$ coordinates is a lattice in $\R^s$. Consider the basis for the subspace $\R^s$ in $\R^{s+t}$:
\begin{equation}\label{eq:basis}
\left\langle
\left(1,0,\ldots,0 , b_{11} , \ldots , b_{1t} \right) ,
\ldots,
\left(0,0,\ldots,1 , b_{s1} , \ldots , b_{st} \right)
\right\rangle .
\end{equation}
Note that, since $\prod_{j=1}^{s+t}\sigma_j(u) = 1$ being equal to the product of the roots of the minimal polynomial of the unit $u$, then for any
$$ \ell(u)=\left( \xi_1, \ldots, \xi_s, \sum_{j=1}^s b_{j1} \xi_j,\ldots, \sum_{j=1}^s b_{jt} \xi_j \right) $$
we have $\sum_{k=1}^{t} \sum_{j=1}^{s} b_{jk} \xi_j = - \sum_{j=1}^{s} \xi_j$; then, for any $j\in\{1,\ldots,s\}$,
$$ \sum_{k=1}^{t} b_{jk} = - 1 . $$
Note in particular that, if $t=1$, then any $b_{j1}=-1$.
Moreover, by definition, $2\lg|\sigma_{s+k}(u)|=\sum_{j=1}^{s}b_{jk}\xi_j$. Set $c_{jk}\in\R$ such that
$$ \sigma_{s+k}(u) = \exp\left(\frac{1}{2}\sum_{j=1}^{s}b_{jk}\xi_j + \sqrt{-1}\sum_{j=1}^{s} c_{jk}\xi_j \right) .$$
Then, we can represent
$$ X(K,U) = \left. \R^s\ltimes_{\varphi}(\R^s\times \C^t) \middle\slash U \ltimes_{\varphi} \mathcal{O}_K \right. $$
where
\begin{equation}\label{eq:varphi}
\varphi(\xi_1,\ldots,\xi_s) =
\left(
\begin{matrix}
\ddots & & & \\
& \exp(\xi_h) & & & \\
& & \ddots & & \\
& & & A_k & \\
& & & & \ddots
\end{matrix}
\right) ,
\end{equation}
where
$$ A_k := \exp\left(\frac{1}{2}\sum_{j=1}^{s}b_{jk}\xi_j + \frac{\sqrt{-1}}{2} \sum_{j=1}^{s} c_{jk} \xi_j \right) .$$
That is, we can identify
\begin{eqnarray*}
\R^s \ltimes_\varphi (\R^s\times\C^t) &=&
\left\{\left(
\begin{matrix}
\ddots & & & & & & &\vdots\\
& \exp(\xi_h) & & & & & &x_h\\
& & \ddots & & &&&\vdots\\
& & & & A_k & & & z_k \\
& & & & & \bar{A}_k && \bar{z}_k&\\
& & & & &   & \ddots&\vdots\\
& & & & &   & &1\\
\end{matrix}
\right) \right. \\[5pt]
&& \left. \;:\; \ldots,\xi_h, \ldots, x_h,\ldots \in \R, \ldots, z_k,\ldots \in \C \right\} .
\end{eqnarray*}

\medskip

We give conditions for which OT manifolds with $t=1$ satisfy Mostow condition; then Proposition \ref{prop:mostow} applies.

\begin{thm}\label{thm:mostow-ot}
Let $X(K,U)$ be an Oeljeklaus-Toma manifold with precisely one complex place. Assume that there is no field $T$ such that $\Q\subset T \subset K$ and $T$ is totally real. Then $X(K,U)$ satisfies the Mostow condition.
\end{thm}

\begin{proof}
Let $X=X(K,U)=\left. \R^s\ltimes_{\varphi}(\R^s\times \C^t) \middle\slash U \ltimes_{\varphi} \mathcal{O}_K \right.$ be an Oeljeklaus-Toma manifold with precisely one complex place, namely $t=1$. In particular, note that any $U$ when $t=1$ is admissible in the sense of \cite{oeljeklaus-toma}. We use notation as described above. We want to prove that $\overline{\mathrm{Ad}\left(\R^s\ltimes_{\varphi}(\R^s\times \C^t)\right)}=\overline{\mathrm{Ad}\left(U \ltimes_{\varphi} \mathcal{O}_K\right)}$ in the Zariski topology of $\mathrm{GL}(\R^{2s+2t})$, where $\mathfrak{g}$ is the Lie algebra of $\R^s\ltimes_{\varphi}(\R^s\times \C^t)$.
In a sense, this extends the criterion of Gorbatsevich from almost Abelian Lie groups to semi-direct products  $\mathbb{R}^s\ltimes_{\varphi}(\R^s\times \C)$.

We first notice that
$$
\overline{\mathrm{Ad}(\mathbb{R}^s\ltimes_{\varphi}(\R^s\times \C))} = \overline{\mathrm{Ad}(\R^s) \ltimes \mathrm{Ad}(\R^s\times \C)} = \overline{ \overline{\mathrm{Ad}(\R^s)}\ltimes \overline{\mathrm{Ad}(\R^s\times \C)} }
$$
and
$$
\overline{\mathrm{Ad}(U \ltimes_{\varphi}\mathcal{O}_K)} = \overline{\mathrm{Ad}(U) \ltimes \mathrm{Ad}(\mathcal{O}_K)} = \overline{ \overline{\mathrm{Ad}(U)}\ltimes \overline{\mathrm{Ad}(\mathcal{O}_K)} } .
$$
This follows by the fact that the Zariski closure of a subgroup of an algebraic group is a subgroup by itself, see {\itshape e.g.} \cite[Proposition I.1.3]{borel}.
Moreover, since $\R^s\times\C$ is the nilradical of $\R^s\ltimes_{\varphi}(\R^s\times \C^t)$, then $\mathrm{Ad}(\R^s\times \C)$ is unipotent and connected, whence Zariski closed, see {\itshape e.g.} \cite[page 2]{raghunathan}. Finally, $\mathrm{Ad}(\mathcal{O}_K)$ is a maximal lattice in $\mathrm{Ad}(\R^s\times\C)$, whence $\overline{\mathrm{Ad}(\mathcal{O}_K)}=\overline{\mathrm{Ad}(\R^s\times\C)}=\mathrm{Ad}(\R^s\times\C)$, see {\itshape e.g.} \cite[Theorem 2.1]{raghunathan}.
At the end, we are reduced to show that $\overline{\mathrm{Ad}(\R^s)}$ and $\overline{\mathrm{Ad}(U)}$ are equal in $\mathrm{GL}(\R^s)$.

Notice that $\mathrm{Ad}((\xi_1, \ldots, \xi_s), 0, \ldots, 0)$ acts trivially on the $\R^s$-component of $\mathfrak{g}$ and as $\varphi(\xi_1, \ldots, \xi_s)$ on the $(\R^s \times \C)$-component, see \eqref{eq:varphi}. Therefore we are reduced to prove that the subgroups $\varphi(\R^s)$ and $\varphi(U)$ have the same Zariski closure in $\mathrm{GL}(\R^{2s+2t})$.

We take $U$ generated by $u_1,\ldots,u_s$ such that
$$ \ldots, \ell(u_h) = \left(t_1^h,\ldots, t_s^h, - (t_1^h + \cdots + t_s^h) \right), \ldots $$
with respect to the basis \eqref{eq:basis}, where $t_j^h\in\R$.
Denote by
$$ R_h := \left(\begin{matrix}
 0 & & & & & & \\
 & \ddots & & & & & \\
 & & 1 & & & & \\
 & & & \ddots & & & \\
 & & & & 0 & & \\
 & & & & & -\frac{1}{2} & c_h \\
 & & & & & -c_h & -\frac{1}{2}
\end{matrix}\right) $$
where the coefficient $1$ is at the intersection between the $h$th row and the $h$th column, (with respect to the notation above, $c_h:=c_{h1}$).
Note that $[R_h,R_m]=0$ for any $h,m\in\{1,\ldots,s\}$.
Denote
$$
G_h := \left\langle \exp\left(\sum_{j=1}^{s} t_j^h R_j\right)\right\rangle , 
\qquad
H_h := \left\langle \exp\left(t \cdot \sum_{j=1}^{s} t_j^h R_j\right)\right\rangle_{t\in\R} .
$$
Then
$$ \varphi(\R^s) = \prod_{j=1}^{s} H_j, \qquad \varphi(U) = \prod_{j=1}^{s} G_h .$$
Arguing as before, $\overline{\prod_{j=1}^{s} H_j}=\overline{\prod_{j=1}^{s} \overline{H}_j}$, and the same for $G_h$, so we are reduced to show that $H_h$ and $G_h$ have the same Zariski closure for any $h\in\{1,\ldots,s\}$.

Each $H_h$ is a $1$-parameter subgroup in $\mathrm{GL}(\R^{s+2})$ and $G_h$ is a discrete subgroup of $H_h$, so the Gorbatsevich criterion in \cite[Lemma 3]{gorbatsevich} applies. We are reduced to show that, for
$$ B^h:=\sum_{j=1}^{s} t_j^h A_j =
\left(
\begin{matrix}
\ddots & & & & \\
& t_j^h & & & \\
& & \ddots & & \\
& & & -\frac{1}{2}\sum_{j=1}^{s} t_j^h & -\sum_{j=1}^{s} c_j t_j^h \\
& & & \sum_{j=1}^{s} c_j t_j^h & -\frac{1}{2}\sum_{j=1}^{s} t_j^h\\
\end{matrix}
\right) , $$
there is no rational linear combination of the eigen-values of $B^h$ equal to $\sqrt{-1}\pi$.

Hereafter, we forget the superscript $h$. The spectrum of $B$ is:
$$ \left\{ t_1, \ldots, t_s, -\frac{1}{2}\sum_{j=1}^{s} t_j+\sqrt{-1}\sum_{j=1}^{s} c_j t_j, -\frac{1}{2}\sum_{j=1}^{s} t_j-\sqrt{-1}\sum_{j=1}^{s} c_j t_j \right\} . $$
Let us assume that there exist $\lambda_1,\ldots,\lambda_s,\eta_1,\eta_2\in\Q$ such that
\begin{eqnarray*}
\sqrt{-1}\pi &=& \sum_{h=1}^{s} \lambda_h t_h + \eta_1 \left( -\frac{1}{2}\sum_{j=1}^{s} t_j+\sqrt{-1}\sum_{j=1}^{s} c_j t_j \right) \\[5pt]
&& + \eta_2 \left( -\frac{1}{2}\sum_{j=1}^{s} t_j-\sqrt{-1}\sum_{j=1}^{s} c_j t_j \right) .
\end{eqnarray*}
Equivalently,
$$
\left\{
\begin{array}{l}
\sum_{h=1}^{s} \lambda_h t_h - \frac{1}{2}\eta_1 \sum_{j=1}^{s} t_j - \frac{1}{2}\eta_2 \sum_{j=1}^{s} t_j = 0 \\[5pt]
\left( \eta_1 - \eta_2 \right) \sum_{j=1}^{s} c_j t_j = \pi
\end{array}
\right.
$$
which yields in particular that the argument of the complex number $\sigma_{s+1}(u_h)$ is $\sum_{j=1}^{s} c_j t_j^{h} = q\pi$ for $q\in\Q$. We are reduced to show that this is not possible.

We first claim that, {\itshape under the assumption that there is no intermediate totally real field $\Q\subset T \subset K$, then $K=\Q(u_h)$, for any $h\in\{1,\ldots,s\}$}. Indeed, we first notice that $\sigma_{s+1}(u_h)\in\C\setminus\R$: otherwise, if $u_h\in\R$, then $\Q(u_h)$ would be a totally real intermediate extension, so $u_h\in\Q$ would be a positive unit; by $U$ being admissible, this is not possible.
Recall that the characteristic polynomial $f_{u_h}$ of $u_h$ is a power of the minimal polynomial $\mu_{u_h}$ of $u_h$, say $f_{u_h}=\mu_{u_h}^k$ for $k\in\N$ (see Proposition 2.6 in \cite{neukirch}). On the other hand, $f_{u_h}(X)=\prod_{j=1}^{s} \left( X-\sigma_j(u_h) \right) \cdot \left( X-\sigma_{s+1}(u_h) \right) \cdot \left( X-\overline{\sigma_{s+1}(u_h)} \right)$ has exactly two complex non-real conjugate roots. Then necessarily $k=1$, that is, $f_{u_h}=\mu_{u_h}$. In particular, $[\mathbb{Q}(u_h):\mathbb{Q}]=\deg\mu_{u_h}= [K : \mathbb{Q}]$, so $K=\mathbb{Q}(u_h)$.

Denote $\alpha_1:=\sigma_1(u_h), \ldots, \alpha_h:=\sigma_s(u_h), \beta:=\sigma_{s+1}(u_h)$, namely, the roots of the minimal polynomial $\mu_{u_h}\in\Z[X]$ of $u_h$. Assume that $\beta$ has argument given by a rational multiple of $\pi$, say, $q\pi$ with $q\in\Q$. Then there exists $N\in\N$ such that $\beta^N=\bar\beta^N$. Since $\beta$ is the root of the monic polynomial $\mu_{u_h}\in\Z[X]$ of degree $s+2$, then there exist $x_{0}, \ldots, x_{s+1}\in\Z$ such that
$$ \beta^N = x_{s+1}\beta^{s+1}+x_{s}\beta^s+\cdots+x_1\beta+x_0 . $$
Set
$$ x := x_{s+1}\beta^{s+1}+x_{s}\beta^s+\cdots+x_1\beta \in \R , $$
such that $\beta^N=\bar\beta^N=x+x_0$. In fact, $x\in\Q$. Indeed, if $x\not\in\Q$, since $\beta\not\in\R$, then $\Q(x)$ would be an intermediate totally real extension $\Q\subseteq\Q(x)\subseteq K=\Q(\beta)$, and it is not possible under the assumption. Consider the polynomial
$$ X^N - (x+x_0)\in \Q[X] . $$
Let $Q(X),R(X)\in\Q[X]$ be such that
$$ X^N - (x+x_0) = Q(X) \cdot \mu_{u_h}(X) + R(X) $$
with $\deg R(X) < s+2$. One has that $R(\beta)=R(\bar\beta)=0$; then $\mu_{u_h}(X)$ divides $R(X)$, with $\deg \mu_{u_h}(X) < \deg R(X)$; then $R(X)=0$. It follows that any $\alpha_j$ is a root of $X^N-(x+x_0)$, that is, $\alpha_1^N=\cdots=\alpha_s^N=\beta^N=\bar\beta^N$. On the other side, recall that $\alpha_1\cdot\cdots\cdot|\beta|^2=1$. It follows that
$$ \left( \alpha_1 \cdot \cdots \cdot \alpha_s \right)^N = \left( \alpha_1 \cdot \cdots \cdot \alpha_s \right)^{-\frac{Ns}{2}} . $$
The $\alpha_j$s being real, this yields
$$ |\beta|^2 = \frac{1}{\alpha_1 \cdot \cdots \cdot \alpha_s} = 1 , $$
that is, $\beta=\exp(\sqrt{-1}q\pi)$.
This says that actually $\beta^N=1$, so any $\alpha_j$ would be a real root of $X^N-1$. But this is not possible, since the $\alpha_j$s are irrational numbers.
\end{proof}

\begin{rmk}
Note in particular that the assumption $s+2$ prime assures that there is no intermediate extension, and so in particular no intermediate totally real extension as required in Theorem \ref{thm:mostow-ot}.

Moreover we show now an explicit example of an Oeljeklaus-Toma manifold $X(K, U)$ of type $(2,1)$ which satisfies the technical condition in Theorem \ref{thm:mostow-ot}. 

Let $f(X) = X^4 -X - 1 \in \Z[X]$; it is irreducible, since its reduction modulo $p=2$ prime, that is, $X^4-X-1\in \Z_2[X]$, is irreducible in $\Z_2[X]$.

{\noindent{\bf Claim 1:}} $f$ has two real roots and two complex (conjugate) roots.\\
Indeed, by Darboux theorem, there is a real root between -1 and 0, so there are at least two real roots. Let $x_1, x_2, x_3, x_4$ be the roots of $f$. By Viette's relations, we have $\sum x_i^2=(\sum x_i)^2-2(\sum_{i\neq j} x_ix_j) = 0$. If all of them were real, then, for all $j$, it holds $x_j=0$. However, but $0$ is not a root of $f$. So two of the roots are real and the other are complex.

Let $\alpha$ be one of the real roots of $f$. Take the algebraic number field $K:=\Q(\alpha)$. Then $\Q \subset K$ is an extension of degree 4, and $X(K, \mathcal{O}_K^{*, +})$ defines an OT manifold of type $(2, 1)$. 

{\noindent{\bf Claim 2:}} $\mathrm{Gal}(f) \simeq S_4$.\\
Indeed, let $\Q_f$ denote the splitting field of $f$ ({\itshape i.e.} the smallest field that contains all the roots of $f$). Note that $\Q_f \neq K$, since  $\Q_f$ contains also complex numbers (namely the complex roots of $f$).
We recall that $\mathrm{Gal}(f) := \{f\colon \Q_f \rightarrow \Q_f \;:\; f(q) = q, \forall q \in \Q \}$.
In \cite[Theorem 7.5.4]{r}, $\mathrm{Gal}(f)$ is explicited for any quartic polynomial $f$. The resolvent of $f$ is the cubic polynomial $q(X) = X^3+4X+1$. As this is an irreducible polynomial over $\Q[X]$ and its discriminant $\Delta = -283$ satisfies $\sqrt{\Delta} \notin \Q$, according to the cited theorem, we have $\mathrm{Gal}(f) \simeq S_4$.

{\noindent{\bf Claim 3:}} There is no intermediate field $\Q \subset T \subset K$.\\
Indeed, let us assume that there exists an intermediate field $\Q \subset T \subset K$. Then we have:
$\Q \subset T \subset K \subset \Q_f$. Since $[\Q_f : \Q]=24$ and $[K: \Q] =4$, $[\Q_f: K]=6$ and we have, in fact, $\mathrm{Gal}(\Q_f/K) \simeq S_3$. This further implies that $S_3 \lneq \mathrm{Gal}(\Q_f/T) \lneq S_4$. However, there is no such intermediate group between $S_3$ and $S_4$, since by a known result in group theory, $S_3$ is a maximal subgroup of $S_4$. Threfore there is no intermediate field between $\Q$ and $K$ and thus, $X(K, \mathcal{O}_K^{*, +})$ satisfies the requirements imposed in Theorem \ref{thm:mostow-ot}.
\end{rmk}

\begin{exa}
For example, for $s=2$ and $t=1$
we choose a co-frame of invariant $1$-forms $\{e^1, e^2, e^3, e^4, e^5,e^6\}$ with structure equations (cf. \cite[Section 6]{kasuya-blms})
$$
\left\{\begin{array}{rcl}
            de^1 &=&          0 \\[5pt]
            de^2 &=& 0 \\[5pt]
            de^3 &=& -e^1 \wedge e^3\\[5pt]
            de^4 &=& -e^2\wedge e^4\\[5pt]
            de^5 &=& \frac{1}{2}e^1 \wedge e^5+c_1 e^1 \wedge e^6+
			\frac{1}{2}e^2 \wedge e^5+c_2 e^2 \wedge e^6 \\[5pt]
            de^6 &=& -c_1 e^1 \wedge e^5+\frac{1}{2}e^1 \wedge e^6-
			c_2 e^2 \wedge e^5+\frac{1}{2}e^2 \wedge e^6
           \end{array}\right. \;,
$$
for some $c_1, c_2\in\R$.
The possible Lee forms of lcs structures are: $e^1+e^2$; and, when $c_1\neq c_2$, also $-e^1-e^2$ (take $\Omega = \omega_{12} e^1\wedge e^2 + \omega_{14}e^1\wedge e^4 + \frac{(4c_1c_2 + 9)\omega_{25} -
6(c_1 - c_2)\omega_{26}}{4c_2^2 + 9}e^1\wedge e^5 + \frac{6(c_1 - c_2)\omega_{25} +
(4c_1c_2 + 9)\omega_{26}}{4c_2^2 + 9}e^1\wedge e^6 + \omega_{23}e^2\wedge e^3 +
\omega_{25}e^2\wedge e^5 + \omega_{26}e^2\wedge e^6 + \omega_{34}e^3\wedge e^4$ for coefficients $\omega_{jk}$ such that $\frac{ 36(c_1-c_2)\cdot (\omega_{25}^2 + \omega_{26}^2)\cdot \omega_{34} }{ 4c_2^2+9 }\neq0$).
The almost-symplectic form
\begin{equation}\label{eq:lcs-OT}
\Omega \;:=\; 2e^1\wedge e^3+ e^1\wedge e^4+e^2\wedge e^3+2e^2\wedge e^4+e^5\wedge e^6
\end{equation}
is a locally conformally symplectic structure with Lee form
$$ \vartheta \;:=\; e^1+e^2 \;. $$
It admits a compatible complex structure $J$:
$$ J e^1 \;:=\; e^3 \;, \quad J e^2 \;:=\; e^4 \quad
J e^3 \;:=\; e^6. 
$$
For suitable values of $c_1$ and $c_2$, by Theorem \ref{thm:mostow-ot} and Proposition \ref{prop:mostow} one can compute the lcs cohomologies of $X(K,U)$.
In Table \ref{table:ot21} we report the dimensions of the Morse-Novikov cohomology groups
(computations have been performed with the help of Sage \cite{sage}.)
Notice that for $k=0$ we recover the Betti numbers of $X(K,U)$ as already computed in \cite[Remark 2.8]{oeljeklaus-toma}.
\end{exa}

\begin{center}
\begin{table}[ht]
 \centering
{\resizebox{\textwidth}{!}{
\begin{tabular}{>{$\mathbf\bgroup}c<{\mathbf\egroup$} | >{$}c<{$} >{$}c<{$} >{$}c<{$} >{$}c<{$} >{$}c<{$} >{$}c<{$} >{$}c<{$} ||}
\toprule
\mathbf{k} & \mathbf{\dim H^0_{d_k}} & \mathbf{\dim H^1_{d_k}} & \mathbf{\dim H^2_{d_k}} & \mathbf{\dim H^3_{d_k}} & \mathbf{\dim H^4_{d_k}} & \mathbf{\dim H^5_{d_k}} & \mathbf{\dim H^6_{d_k}} \\
\toprule
k=-1 & 0 & 0 & 1 & 2 & 1 & 0 & 0 \\
k=0 & 1 & 2 & 1 & 0 & 1 & 2 & 1 \\
k=1 & 0 & 0 & 1 & 2 & 1 & 0 & 0 \\
\midrule[0.02em]
\bottomrule
\end{tabular}
}}
\caption{Summary of the dimensions of the Morse-Novikov cohomologies of an Oeljeklaus-Toma manifold of type $(2,1)$. (Just non-trivial cohomology groups are reported.)}
\label{table:ot21}
\end{table}
\end{center}

More in general, we show that Oeljeklaus-Toma manifolds of type $(s,1)$ that satisfy the technical condition in Theorem \ref{thm:mostow-ot} can be found for any $s\geq 1$.

\begin{prop}\label{prop:example-mostow}
Let $s>0$ be a natural number. Then there exists $K$ an algebraic number field with $s$ real embeddings and $2$ conjugate complex embeddings such that there is no intermediate extension between $\Q$ and $K$.
\end{prop}

\begin{proof} Let $n=s+2$.  The idea is to prove the existence of a monic irreducible polynomial $f \in \Z[X]$ of degree $s+2$ such that $f$ has $s$ real roots, $2$ conjugate complex roots and $\mathrm{Gal}(f)=S_n$. Once proven this, take $K=\Q(\alpha)$, where $\alpha$ is one of the roots of $f$. Like in the example, we would have $\mathrm{Gal}(\Q(f)/K)=S_{n-1}$. The existence of an intermediate field between $\Q$ and $K$ would  imply the existence of a subroup $H$ of $S_n$, such that $S_{n-1} \lneq H \lneq S_n$. But this does not exist, as $S_{n-1}$ is a maximal subgroup of $S_n$.

A construction of a polynomial $f$ whose $\mathrm{Gal}(f)$ is $S_n$ was given by B.L. van der Waerden. The idea was to consider the following monic polynomial $f = - 15 f_1+10f_2+6f_3$, where $f_1$, $f_2$ and $f_3$ are degree $n$ polynomials and $f_1$ reduced in $\Z_2[X]$ is irreducible, $f_2$ decomposes in $\Z_3[X]$ as a product of a linear factor and a degree $n-1$ irreducible polynomial, and $f_3$ decomposes in $\Z_5$ as a product of an irreducible quadratic polynomial and a degree $n-2$ irreducible polynomial, if $n$ is odd, or as a product of an irreducible quadratic polynomial and two irreducible polynomials of odd degree, if $n$ is even. It is explained in Proposition 4.7.10 in \cite{weintraub} why there exist $f_1$, $f_2$ and $f_3$ with these properties and why $f$ thus defined has Galois group $S_n$. Observe that $f$ is irreducible because we have $f=f_1$ modulo $2$, which is irreducible in $\Z_2[X]$. Morever, if $g$ is any polynomial of degree $n-1$, then $f+30g$ is also an irreducible polynomial with Galois group $S_n$.

Now we use the same argument as in Remark 1.1 in \cite{oeljeklaus-toma}. Namely, let $D=\{(a_1, \ldots, a_n)\} \subseteq \Q^n$ be the set of $n$-uples such that $h=X^n+a_1X^{n-1}+\cdots+a_n$ (not necessarily irreducible) has $s$ real roots and $2$ complex roots. Then $D$ is a non-empty set which contains arbitrarily large open balls, as argumented in \cite{oeljeklaus-toma}.
If $f=X^n+b_1X^{n-1}+\cdots+b_n$, consider the set $D'=(b_1, b_2, \ldots , b_n)+30\Z^n$. Then $D'$ intersects $D$ and the intersection consists of irreducible polynomials with $s$ real roots, $2$ complex roots and Galois group $S_n$. 
\end{proof}

As a corollary we obtain:

\begin{cor}
For any natural number $s \geq 1$, we obtain an Oeljeklaus-Toma manifold of type $(s, 1)$ satisfying the Mostow condition.
\end{cor}

\FloatBarrier

\end{document}